\documentclass{article}
\usepackage{arxiv}
\usepackage{amsmath,amsthm,amsfonts,amssymb,enumerate}
\usepackage{geometry}
\usepackage{algorithm}
\usepackage{multirow}
\usepackage[utf8]{inputenc} 
\usepackage[T1]{fontenc}    
\usepackage{hyperref}       
\usepackage{url}            
\usepackage{booktabs}       
\usepackage{amsfonts}       
\usepackage{nicefrac}       
\usepackage{microtype}      
\usepackage{cleveref}       
\usepackage{lipsum}         
\usepackage{graphicx}
\usepackage{natbib}
\usepackage{doi}
\newtheorem{theorem}{Theorem}
\newtheorem{proposition}{Proposition}
\newtheorem{definition}{Definition}
\newtheorem{lemma}{Lemma}
\title{Parameterized proximal-gradient algorithms for $L{1}/L{2}$ sparse signal recovery\thanks{This work was supported by the National Science Foundation of China under grant 11971499, and by the Opening Project of Guangdong Province Key Laboratory of Computational Science at the Sun Yat-sen University under grant 2021001, and by the Guangdong Province Key Laboratory of Computational Science at Sun Yat-sen University under grant 2020B1212060032, and by the Natural Science Foundation of Shandong Province, China under grant ZR2021QA104.}}

\author{ {Na Zhang}\thanks{Department of Applied Mathematics, College of Mathematics and Informatics, South China Agricultural University, Guangzhou 510642, P. R. China. } \\
	\And
	{Xinrui Liu}\thanks{Corresponding author. School of Mathematics and Statistics, Shandong Normal University, Jinan 250358, P. R. China (xinrui23@sdnu.edu.cn).} \\
        \And
	{Qia Li}\thanks{Guangdong Province Key Laboratory of Computational Science, School of Computer Science and Engineering, Sun Yat-sen University, Guangzhou 510275, P. R. China.} \\
}

\renewcommand{\shorttitle}{\textit{arXiv} Template}

\hypersetup{
pdftitle={A template for the arxiv style},
pdfsubject={q-bio.NC, q-bio.QM},
pdfauthor={David S.~Hippocampus, Elias D.~Striatum},
pdfkeywords={First keyword, Second keyword, More},
}

\begin{document}
\maketitle

\begin{abstract}
The ratio of $\ell_{1}$ and $\ell_{2}$ norms ($\ell_{1}/\ell_{2}$), serving as a sparse promoting function, receives considerable attentions recently due to its effectiveness for sparse signal recovery. In this paper, we propose an $\ell_{1}/\ell_{2}$ based penalty model for recovering sparse signals from noiseless or noisy observations. It is proven that  stationary points of the proposed problem tend to those of the elliptically constrained $\ell_{1}/\ell_{2}$ minimization problem as the smoothing parameter goes to zero. Moreover, inspired by the parametric approach for the fractional programming, we design a parameterized proximal-gradient algorithm (PPGA) as well as its line search counterpart (PPGA$\_$L) for solving the proposed model. The closed-form solution of the involved proximity operator is derived, which enable the efficiency of the proposed algorithms. We establish the global  convergence of the entire sequences generated by PPGA and PPGA$\_$L with monotone objective values by taking advantage of the fact that the objective of the proposed model is a KL function. Numerical experiments show the efficiency of the proposed algorithms over the state-of-the-art methods in both noiseless and noisy sparse signal recovery problems. 
\end{abstract}



\section{Introduction}\label{S1}
Compressed sensing (CS) is an advance in signal acquisition and attracts much attention in many research areas, such as signal processing, image processing and information theory \cite{duarte2013spectral,laska2011democracy,Tao2006Stable,Chun2014Non,Donoho2006Compressed}. Let $x_{0}\in\mathbb{R}^{n}$ be a sparse or approximately sparse signal and $A\in\mathbb{R}^{m\times n} (m<<n)$ be the sensing matrix. CS aims at recovering the sparse signal $x_{0}$ from a small number of its possibly noisy measurements $b = Ax_{0}+e$, where $e\in\mathbb{R}^{m}$ is the noise vector.

When there is no noise in the transmission, i.e., $b = Ax_{0}$, the CS problem can be mathematically formulated as 
\begin{equation}\label{model_origin}
{\rm min}\left\{\|x\|_{0}:Ax = b,x\in\mathbb{R}^{n}\right\},
\end{equation} 
where $\|\cdot\|_{0}$ is the $\ell_{0}$ ``norm'' which counts the number of nonzero entries of the vector. Since (\ref{model_origin}) is NP-hard \cite{Natarajan1995Sparse} to solve, a popular approach in CS is to replace $\ell_{0}$ norm by the convex $\ell_{1}$ norm, and then many efficient algorithms were designed for solving the problem \cite{Chen2001Atomic,Figueiredo2007Gradient,Goldstein2009split,Tendero2021Algorithm}. It has been demonstrated that a sparse signal $x_{0}$ can be recovered exactly by minimizing the $\ell_{1}$ norm over $A^{-1}\{b\}$ if $x_{0}$ is sufficiently sparse and the matrix $A$ satisfies the restricted isometry property (RIP) \cite{Candes2005Decoding,Chen2001Atomic}. In recent years, nonconvex approximation functions of the $\ell_{0}$ norm attract considerable amount of attentions due to their shaper approximations to $\ell_{0}$ compared with the $\ell_{1}$ norm. Some popular nonconvex functions include $\ell_{p}(0<p<1)$ \cite{Antil2022Sparse,Chartrand2007Exact,Lai2013Improved,Xu2012regularization}, capped $\ell_{1}$ \cite{Bian2020smoothing,Li2013Capped}, transformed $\ell_{1}$ \cite{Natarajan1995Sparse,Zhang2017Minimization,Zhang2018Minimization}, $\ell_{1}-\ell_{2}$ \cite{BI2022337,Lou2015Computing,Yin2014Ratio} and so on.

We study in this paper the ratio of the $\ell_{1}$ and $\ell_{2}$ norms ($\ell_{1}/\ell_{2}$) to approximate the desired $\ell_{0}$ norm. The $\ell_{1}/\ell_{2}$ function may date back to \cite{Hoyer2002coding} as a sparsity measure. It was explicitly pointed out \cite{Hurley2009Comparing} that similar with the $\ell_{0}$ norm, $\ell_{1}/\ell_{2}$ is a scale-invariant measure which may allow the $\ell_{1}/\ell_{2}$ based model to reconstruct signals and images \cite{Wang2022Minimizing,Wang2021Limited} with high dynamic range. In order to show the advantage of the $\ell_{1}/\ell_{2}$ function, the authors of \cite{Rahimi-Wang-Dong-Lou} also exhibited a specific example in which the ground truth signal $x_{0}$ is a solution of 
\begin{equation}\label{model_ratio}
{\rm min}\left\{\frac{\|x\|_{1}}{\|x\|_{2}}:Ax = b,x\in\mathbb{R}^{n}\right\},
\end{equation}
while $x_{0}$ is not a solution of the $\ell_{1}$, $\ell_{p}(p = 1/2)$ or $\ell_{1}-\ell_{2}$ based models. The equivalence between $\ell_{1}/\ell_{2}$ and $\ell_{0}$ for reconstructing signals under suitable conditions was studied in \cite{Esser2013method,Xu2021Analysis,Yin2014Ratio,Zhou2021Minimization}. Since $\ell_{1}/\ell_{2}$ is nonconvex and nonsmooth, it is difficult to find a global minimizer of problem (\ref{model_ratio}) and many algorithms are proposed to search for a stationary point of problem (\ref{model_ratio}). In \cite{Rahimi-Wang-Dong-Lou}, the authors tried to utilize the alternating direction method of multipliers (ADMM) to solve problem (\ref{model_ratio}) and showed its efficiency empirically in numerical simulations. It has been proven in \cite{Wang2021Limited} that under mild conditions the sequence generated by ADMM has a subsequence convergent to a stationary point of (\ref{model_ratio}). Three accelerated algorithms for solving (\ref{model_ratio}), namely, the $\ell_{1}/\ell_{2}$ minimization via bisection search (BS) and the $\ell_{1}/\ell_{2}$ minimization via adaptive selection method (A1 or A2), were proposed in \cite{Chao2019Accelerated} and their subsequential convergence were established.  

Generally, there is noise in the measurement in practice. When the noise is assumed to be white Gaussian noise, it is customary to relax the equality constraint $Ax = b$ to an inequality constraint \cite{Tao2006Stable} $\|Ax-b\|_{2}\le \epsilon$ for $0<\epsilon<\|b\|_{2}$. Thus, the $\ell_{1}/\ell_{2}$ based model for reconstructing sparse signals from the noisy measurements can be formulated as
\begin{equation}\label{model_ratio_constraint}
{\rm min}\left\{\frac{\|x\|_{1}}{\|x\|_{2}}:\|Ax-b\|_{2}\le \epsilon,x\in\mathbb{R}^{n}\right\}.
\end{equation}
The existence of solutions of problem  (\ref{model_ratio_constraint}) can be theoretically guaranteed \cite{Zeng2021Analysis} by assuming the kernel of $A$ has the s-spherical section property \cite{Vavasis2009Derivation,Zhang2013Theory} and there exists $\widetilde{x}\in\mathbb{R}^{n}$ such that $\|\widetilde{x}\|_{0}<m/s$ and $\|A\widetilde{x}-b\|_{2}\le\epsilon$ for some $s>0$.
Although algorithms for solving (\ref{model_ratio_constraint}) with $\ell_{1}$ norm or $\ell_{p}(0<p<1)$ quasi-norm in place of $\ell_{1}/\ell_{2}$ have been extensively studied in \cite{Chen2016Penalty,She2016dual,Berg2008Probing}, these algorithms are not directly applicable for solving (\ref{model_ratio_constraint}) due to the fractional objective. In \cite{Zeng2021Analysis}, the authors applied the moving-balls-approximation (MBA) techniques to solve problem (\ref{model_ratio_constraint}) and showed the global sequential convergence of the proposed algorithm.  Another popular model to deal with noisy compressed sensing is  the  $\ell_{1}/\ell_{2}$ regularized least squares problem, whose objective is the sum of $\ell_1/\ell_2$ and the square of the distance between $Ax$ and $b$. Very recently, \cite{TaoMinimization2022} derived the analytic solution of the proximity operator of $\ell_1/\ell_2$ which relies on two unknown values related to solutions of the problem. Then, by appropriately splitting the variable into two blocks, nonconvex ADMM with global convergence was applied to solve the $\ell_{1}/\ell_{2}$ least squares problem\cite{TaoMinimization2022}.

In this paper, we focus on  the $\ell_{1}/\ell_{2}$ based signal recovery from noiseless or noisy measurements. First, it should be noted that the optimal value of problem (\ref{model_ratio_constraint}) may be not attainable for some $A, b$ and $\epsilon$. For instance, let $A$ be a $2\times 3$ matrix whose first row is $(1,0,0)$ and second row is $(0,1,0)$, $b=(1,1)^\top$ and $\epsilon=0$. One can check that the optimal value of problem(\ref{model_ratio_constraint}) is 1. However, the objective value is strictly bigger than 1 at any feasible point of the problem. In order to ensure the existence of optimal solutions of the problem as long as the constraint set is nonempty,  it is straightforward to  restrict the variable in a ball. This leads to the following problem      
 
\begin{equation}\label{l1/l2_minimization_problem_constraint}
{\rm min}\left\{\frac{\|x\|_{1}}{\|x\|_{2}}:\|Ax-b\|_{2}\le \epsilon,\|x\|_{2}\le d, x\in\mathbb{R}^{n}\right\},
\end{equation}
where $d>0$. We also note that the parameter $\epsilon$ may impact the  solution of problem (\ref{l1/l2_minimization_problem_constraint}) significantly. Particularly, if $A$ does not have full row rank and $b$ is not in the range of $A$, the constraint set of problem (\ref{l1/l2_minimization_problem_constraint}) will be empty when $\epsilon$ is small enough. To overcome this deficiency, we propose to replace the elliptic constraint $\{x\in\mathbb{R}^n: \|Ax-b\|_2\le \epsilon\}$ with a smooth fractional penalty function and then obtain an $\ell_1/\ell_2$ penalty problem, whose optimal value is definitely attainable. Actually, the proposed $\ell_{1}/\ell_{2}$ penalty problem can be viewed as an approximate model to problem (\ref{l1/l2_minimization_problem_constraint}) through smoothing the nonsmooth indicator function on the elliptic constraint by a smooth fractional function. We  prove that as the smoothing parameter goes to zero, stationary points of the proposed penalty problem tends to those of problem  (\ref{l1/l2_minimization_problem_constraint}). Observing that the objective function of the proposed penalty problem is a ratio of two functions, we develop a parameterized proximal-gradient algorithm and its line search counterpart for solving the problem, inspired by the parametric approach for the fractional programming.  Numerical experiments show that on noise-free cases, our proposed algorithms perform comparable with the state-of-the-art algorithms for solving $\ell_{1}/\ell_{2}$ signal recovery model in terms of the success rate. For noisy data, on one hand, the proposed algorithms can obtain competitive recovery results but use much less CPU time compared with MBA. On the other hand, compared with other sparse promoting models, our proposed algorithms perform  better in most cases than the competing algorithms in terms of the signal recovery accuracy and computing time. The main contributions of this paper are summarized below:

(i) In order to recover sparse signals from noiseless or noisy measurements, we propose an $\ell_1/\ell_2$ based penalty problem, which is derived by replacing the elliptic constraint of problem (\ref{l1/l2_minimization_problem_constraint}) with a smooth fractional penalty function.  We prove that as the smoothing parameter goes to zero, the stationary points of the proposed penalty problem tend to those of problem (\ref{l1/l2_minimization_problem_constraint}).

(ii) We study a general structured single-ratio fractional optimization problem, of which the proposed penalty problem is a special case, and propose a parameterized proximal-gradient algorithm and its line search counterpart for solving it. When the proposed algorithms are applied to the proposed model, the analytic solution of the proximity operator of the corresponding function is derived, which enable the efficiency of the algorithms.

(iii) We establish that the entire sequence generated by the proposed parameterized proximal-gradient algorithm or its monotone line search counterpart converges globally to a stationary point of the proposed problem. Numerical simulations demonstrate the efficiency of the proposed algorithms.

The remaining part of this paper is organized as follows. In section \ref{S2}, we introduce notation and some necessary preliminaries. Section \ref{S3} is devoted to the proposed model. In section \ref{S4}, we propose the parameterized proximal-gradient algorithm and its line search counterpart for solving the proposed problem. The convergence results of the proposed algorithms are established in section \ref{S5}. We present in section \ref{S6} some numerical results for $\ell_{1}/\ell_{2}$ sparse signal recovery problem with and without noise. Finally, we conclude this paper in the last section.
\section{Notation and preliminaries}\label{S2}
We start by our prefered notation. Let $\mathbb{R}^{n}$ and $\mathbb{N}$ be the n-dimensional Euclidean space and the set of nonnegative integers respectively. By $\langle\cdot,\cdot\rangle$, $\|\cdot\|_{1}$ and $\|\cdot\|_{2}$, we denote the standard inner product, the $\ell_{1}$ norm and $\ell_{2}$ norm of vectors, respectively. For a positive integer $n$, we let $0_{n}$ be the $n$-dimensional zero vector. For $x\in\mathbb{R}$, let $[x]_{+}:={\rm max}\{0,x\}$ and $\lceil x\rceil:={\rm max}\{t\in\mathbb{N}:t\le x\}$. 

For any closed set $S\subseteq\mathbb{R}^{n}$, the distance from $x\in\mathbb{R}^{n}$ to $S$ is defined by ${\rm dist}(x,S):={\rm inf}\{\|x-z\|_{2}:z\in S\}$. We denoted by ${\rm int}S$ the set of all interior points of $S$. The indicator function on $S$ is defined by 
$$
\iota_{S}(x):= \left\{
\begin{aligned}
&0,  \quad  \quad \quad{\rm if} \ x\in S, \\
&+\infty, \quad \ {\rm otherwise}.
\end{aligned}
\right.
$$
Given a function $\varphi:\mathbb{R}^{n}\rightarrow \overline{\mathbb{R}}:=(-\infty,+\infty]$, we use ${\rm dom}(\varphi)$ to denote the domain of $\varphi$, that is, ${\rm dom}(\varphi):=\{x\in\mathbb{R}^{n}:\varphi(x)<+\infty\}$. The proximity operator of $\varphi$, denoted by ${\rm prox}_{\varphi}$, is defined at $x\in\mathbb{R}^{n}$ as ${\rm prox}_{\varphi}(x):={\rm arg\,min}\{\varphi(y)+\frac{1}{2}\|y-x\|_{2}^{2}:y\in\mathbb{R}^{n}\}$.

In the remaining part of this section, we present some preliminaries on Fr{\'e}chet subdifferential and limiting-subdifferential as well as the Kurdyka-\L ojasiewicz (KL) property. These concepts play an important role in our theoretical and algorithmic development.
\subsection{Fr{\'e}chet subdifferential and limiting-subdifferential} 
Let $\varphi:\mathbb{R}^{n}\rightarrow \overline{\mathbb{R}}$ be a proper function. The Fr{\'e}chet subdifferential of $\varphi$ at $x\in{\rm dom}(\varphi)$, denoted by $\hat{\partial}\varphi(x)$, is defined as 
$$
\hat{\partial}\varphi(x):=\left\{y\in\mathbb{R}^{n}:\mathop{\rm lim\,inf}\limits_{z\rightarrow x, z\neq x}\frac{\varphi(z)-\varphi(x)-\langle y,z-x\rangle}{\|z-x\|_{2}}\right\}.
$$
The set $\hat{\partial}\varphi(x)$ is convex and closed for $x\in{\rm dom}(\varphi)$. If $x\notin{\rm dom}(\varphi)$, we let $\hat{\partial}\varphi(x) = \emptyset$. We say $\varphi$ is Fr{\'e}chet subdifferentiable at $x\in\mathbb{R}^{n}$ when $\hat{\partial}\varphi(x)\neq 0$. We also need the notion of limiting-subdifferentials. The limiting-subdifferential of $\varphi$ at $x\in{\rm dom}(\varphi)$ is defined by 
$
\partial\varphi(x):=\{y\in\mathbb{R}^{n}:\exists x^{k}\rightarrow x, \varphi(x^{k})\rightarrow \varphi(x), y^{k}\in\hat{\partial}\varphi(x^{k})\rightarrow y\}.
$
It is straightforward that $\hat{\partial}\varphi(x) \subseteq \partial\varphi(x)$ for all $x\in\mathbb{R}^{n}$. We say $\varphi$ is regular at $x\in{\rm dom}(\varphi)$ if $\hat{\partial}\varphi(x) = \partial\varphi(x)$. Moreover, if $\varphi$ is convex, then $\hat{\partial}\varphi(x)$ and $\partial\varphi(x)$ reduce to the classical subdifferential in convex analysis, i.e.,
$$
\hat{\partial}\varphi(x) = \partial\varphi(x) = \{y\in\mathbb{R}^{n}: \varphi(z)-\varphi(x)-\langle y,z-x\rangle\ge 0, \forall z\in\mathbb{R}^{n}\}.
$$

We next recall some simple and useful calculus results on $\hat{\partial}$ and $\partial$. For any $\alpha>0$ and $x\in\mathbb{R}^{n}$, $\hat{\partial}(\alpha\varphi)(x) = \alpha\hat{\partial}\varphi(x)$ and $\partial(\alpha\varphi)(x) = \alpha\partial\varphi(x)$. Let $\varphi_{1}, \varphi_{2}: \mathbb{R}^{n}\rightarrow \overline{\mathbb{R}}$ be proper and lower semicontinuous and let $x\in {\rm dom}(\varphi_{1}+\varphi_{2})$, then, $\hat{\partial}\varphi_{1}(x)+\hat{\partial}\varphi_{2}(x)\subseteq \hat{\partial}(\varphi_{1}+\varphi_{2})(x)$. If $\varphi_{1}$ is strictly continuous at $x$ while $\varphi_{2}(x)$ is finite, then $\partial(\varphi_{1}+\varphi_{2})(x)\subseteq \partial\varphi_{1}(x)+\partial\varphi_{2}(x)$. If $\varphi_{2}$ is differentiable at $x$, then $\hat{\partial}\varphi_{2}(x) = \{\nabla\varphi_{2}(x)\}$ and $\hat{\partial}(\varphi_{1}+\varphi_{2})(x)= \hat{\partial}\varphi_{1}(x)+\nabla\varphi_{2}(x)$. Furthermore, if $\varphi_{2}$ is continuously differentiable at $x$,
then $\partial\varphi_{2}(x) = \{\nabla\varphi_{2}(x)\}$ and $\partial(\varphi_{1}+\varphi_{2})(x)= \partial\varphi_{1}(x)+\nabla\varphi_{2}(x)$.
We next recall some results of the Fr{\'e}chet subdifferential for the quotient of two functions. 
\begin{proposition}\label{P2.1}\cite{Na2020First}(Subdifferential calculus for quotient of two functions). Let $f_{1}:\mathbb{R}^{n}\rightarrow \overline{\mathbb{R}}$ be proper and $f_{2}:\mathbb{R}^{n}\rightarrow  \mathbb{R}$. Define $\rho:\mathbb{R}^{n}\rightarrow  \overline{\mathbb{R}}$ at $x\in\mathbb{R}^{n}$ as
\begin{equation}\label{subdiff_quo}
\rho(x):=\left\{
\begin{aligned}
&\frac{f_{1}(x)}{f_{2}(x)}, \quad {\rm if} \ x\in{\rm dom}(f_{1}) \ {\rm and} \ f_{2}(x) \neq 0,\\
&+\infty, \quad\  {\rm otherwise}.
\end{aligned}
\right.
\end{equation}
Let $x\in{\rm dom}(\rho)$ with $a_{1}:=f_{1}(x)\ge 0$ and $a_{2}:= f_{2}(x) > 0$. If $f_{1}$ is continuous at $x$ relative to ${\rm dom}(f_{1})$ and $f_{2}$ is locally Lipschitz continuous at $x$, then
$$
\hat{\partial}\rho(x) = \frac{\hat{\partial}(a_{2}f_{1}-a_{1}f_{2})(x)}{a_{2}^{2}}.
$$
Furthermore, if $f_{2}$ is differentiable at $x$, then
$$
\hat{\partial}\rho(x) = \frac{a_{2}\hat{\partial}f_{1}(x)-a_{1}\nabla f_{2}(x)}{a_{2}^{2}}.
$$
\end{proposition}
\subsection{Kurdyka-\L ojasiewicz (KL) property}
We next recall the Kurdyka-\L ojasiewicz (KL) property.
\begin{definition} (KL property \cite{Hedy2009On}). A proper function $\varphi:\mathbb{R}\rightarrow  \overline{\mathbb{R}}$ is said to satisfy the KL property at $\hat{x}\in{\rm dom}(\partial\varphi)$ if there exist $\eta\in (0,+\infty]$, a neighborhood $\mathcal{O}$ of $\hat{x}$ and a continuous concave function $\varphi: [0,\eta)\rightarrow [0,+\infty]$, such that:\\
(i) $\varphi(0) = 0$,\\
(ii) $\varphi$ is continuously differentiable on $(0,\eta)$ with $\varphi^{'}>0$,\\
(iii) For any $x\in \mathcal{O} \cap \{x\in\mathbb{R}^{n}: \varphi(\hat{x})<\varphi(x)<\varphi(\hat{x})+\eta\}$ there holds $\varphi^{'}(\varphi(x)-\varphi(\hat{x})){\rm dist}(0,\partial\varphi(x))$\\$\ge 1$.
\end{definition}

A proper lower semicontinuous function $\varphi:\mathbb{R}\rightarrow  \overline{\mathbb{R}}$ is called a KL function if $\varphi$ satisfies the KL property at all points in ${\rm dom}(\partial\varphi)$. The notion of the KL property plays a crucial rule in analyzing the global sequential convergence; see \cite{Hedy2009On,Hedy2010Proximal,Hedy2013Convergence}. A framework for proving global sequential convergence using the KL property is provided in \cite{Hedy2013Convergence}. We review this result in the next proposition.

\begin{proposition}\label{KL_property}
Let $\varphi:\mathbb{R}\rightarrow  \overline{\mathbb{R}}$ be a proper lower semicontinuous function. Consider a sequence satisfies the following three conditions:\\
(i) (Sufficient decrease condition.) There exists $a > 0$ such that
$
\varphi(x^{k+1}) + a\|x^{k+1}-x^{k}\|_{2}^{2}\le \varphi(x^{k})
$
holds for any $k\in\mathbb{N}$;\\
(ii) (Relative error condition.) There exists $b > 0$ and $w^{k+1}\in\partial\varphi(x^{k+1})$ such that
$
\|w^{k+1}\|_{2}\le b\|x^{k+1}-x^{k}\|_{2}
$
holds for any $k\in\mathbb{N}$;\\
(iii) (Continuity condition.) There exists a subsequence $\{x^{k_{j}}: j \in \mathbb{N}\}$ and $x^{\star}$ such that
$
x^{k_{j}} \rightarrow x^{\star} \ {\rm and}\  \varphi(x^{k_{j}}) \rightarrow \varphi(x^{\star}), \ {\rm as}\ j\rightarrow \infty.
$
If $\varphi$ satisfies the KL property at $x^{\star}$, then $\sum_{k=1}^{\infty}\|x^{k}-x^{k-1}\|_{2}<+\infty$, $\mathop{\rm lim}\limits_{k\rightarrow \infty}x^{k} = x^{\star}$ and $0\in\partial\varphi(x^{\star})$.
\end{proposition}

\section{The proposed model}\label{S3}
As discussed in section \ref{S1}, the constraint set of problem (\ref{l1/l2_minimization_problem_constraint}) may be empty for some cases when $\epsilon$ is small enough. In this section, we first propose an $\ell_1/\ell_2$ based penalty problem by replacing the elliptic constraint of problem (\ref{l1/l2_minimization_problem_constraint}) with a smooth fractional penalty function. Then, we prove under mild conditions that stationary points of the proposed penalty problem tend to those of problem (\ref{l1/l2_minimization_problem_constraint}) as the smoothing parameter goes to zero.

Before introducing our proposed  model, we first recall the notation of Moreau envelope, which is a useful tool for smoothing functions. The Moreau envelope of a function $\phi:\mathbb{R}^{n}\rightarrow \bar{\mathbb{R}}$ with parameter $\lambda>0$, denoted by ${\rm env}_{\phi}^{\lambda}$, is defined at $x\in \mathbb{R}^{n}$ as
$$
{\rm env}_{\phi}^{\lambda}(x):= {\rm min}\left\{\phi(y) + \frac{1}{2\lambda}\|x-y\|_{2}^{2}: y\in\mathbb{R}^{n}\right\}.
$$
Let $C: = \{x\in\mathbb{R}^{m}:\|x-b\|_{2}\le \epsilon\}$. It is observed that, for $x\in\mathbb{R}^{m}$,
$$
\begin{aligned}
{\rm env}_{\iota_{C}}^{\lambda}(x) &= {\rm min}\left\{\frac{1}{2\lambda}\|x-y\|_{2}^{2}:\|y-b\|_{2}\le \epsilon \right\}\\
&= \left\{
\begin{aligned}
&0, \qquad \qquad \quad \ \, \qquad {\rm if} \ \|x-b\|_{2}\le \epsilon, \\
&\frac{1}{2\lambda}(\|x-b\|_{2}-\epsilon)^{2}, \ {\rm otherwise}, 
\end{aligned}
\right.\\
&=\frac{1}{2\lambda}(\|x-b\|_{2}-\epsilon)_{+}^{2}.
\end{aligned}
$$
It is known  \cite{Na2017On} that ${\rm env}_{\iota_{C}}^{\lambda}\circ A$ is Lipschitz continuously differentiable with a Lipschitz constant $\frac{1}{\lambda}\|A\|_{2}^{2}$ and 
$$
\begin{aligned}
\nabla({\rm env}_{\iota_{C}}^{\lambda}\circ A)(x) &= \frac{(\|Ax-b\|_{2}-\epsilon)_{+}}{\|Ax-b\|_{2}}A^{T}(Ax-b)/\lambda \\ &=\left\{
\begin{aligned}
&0_{n}, \qquad \qquad \qquad \qquad \qquad \qquad \quad {\rm if} \ \|Ax-b\|_{2}\le \epsilon, \\
&(1-\frac{\epsilon}{\|Ax-b\|_{2}})A^{T}(Ax-b)/\lambda, \quad  {\rm otherwise}.
\end{aligned}
\right.
\end{aligned}
$$

 Noting that $({{\rm env}_{\iota_{C}}^{\lambda}\circ A})/{\|\cdot\|_{2}}$ is continuously differentiable on $\mathbb{R}^{n}\backslash\{0_{n}\}$, we propose the following optimization problem for sparse signal recovery through smoothing $\iota_C\circ A$ of problem  (\ref{l1/l2_minimization_problem_constraint}) by $({{\rm env}_{\iota_{C}}^{\lambda}\circ A})/{\|\cdot\|_{2}}$
\begin{equation}\label{l1_l2_minimization_problem_constraint_envelop}
{\rm min}\left\{\frac{\lambda\|x\|_{1}+{\rm env}_{\iota_{C}}^{1}(Ax)}{\|x\|_{2}}:x\in D, x\neq 0_{n}\right\},
\end{equation}
where $D:=\{x\in\mathbb{R}^{n}:\|x\|_{2}\le d\}$. Problem (\ref{l1_l2_minimization_problem_constraint_envelop}) is actually a penalty problem for (\ref{l1/l2_minimization_problem_constraint}), utilizing $({{\rm env}_{\iota_{C}}^{\lambda}\circ A})/{\|\cdot\|_{2}}$ as the penalty function.   In the remaining part of this section, we dedicate to establishing the relationship between solutions and stationary points of problems  (\ref{l1/l2_minimization_problem_constraint}) and (\ref{l1_l2_minimization_problem_constraint_envelop}). Prior to that, we first show the existence of optimal solutions to problems (\ref{l1/l2_minimization_problem_constraint}) and (\ref{l1_l2_minimization_problem_constraint_envelop}). For convenient of presentation, we define $Q_{\lambda}:\mathbb{R}^{n}\rightarrow \bar{\mathbb{R}}$ at $x\in\mathbb{R}^{n}$ as 
$$
Q_{\lambda}(x):= \left\{
\begin{aligned}
&\frac{\lambda\|x\|_{1}+{\rm env}_{\iota_{C}}^{1}(Ax)}{\|x\|_{2}}, \qquad  {\rm if} \ x\in D \ {\rm and} \ x\neq 0_{n}, \\
&+\infty, \qquad \qquad\qquad \qquad \, {\rm else}.
\end{aligned}
\right.
$$ 
The following Lemma concerns the lower semicontinuity of function $Q_{\lambda}$ for any $\lambda>0$.
\begin{lemma}\label{lower_semicontinuous}
For any $\lambda>0$, $Q_{\lambda}$ is lower semicontinuous.
\end{lemma}
\begin{proof}
Let $x\in \mathbb{R}^{n}$. If $x\in {\rm dom}(Q_{\lambda})$, one has $\mathop{\rm lim\ inf}\limits_{y\rightarrow x} Q_{\lambda}(y) = Q_{\lambda}(x)$ due to the continuity of $\|\cdot\|_{1}$, ${\rm env}_{\iota_{C}}^{1}$ and $\|\cdot\|_{2}$. If $x\notin {\rm dom}(Q_{\lambda})$, then $x\notin D$ or $x=0_{n}$. For the case $x\notin D$, $\mathop{\rm lim\ inf}\limits_{y\rightarrow x} Q_{\lambda}(y) = +\infty = Q_{\lambda}(x)$ because $D$ is a closed set. When $x=0_{n}$, it is obvious that $\mathop{\rm lim}\limits_{y\rightarrow x}\lambda\|y\|_{1} + {\rm env}_{\iota_{C}}^{1}(Ay) = \frac{1}{2}(\|b\|_{2}-\epsilon)_{+}^{2}>0$ from the assumption that $\|b\|_{2}>\epsilon$. Thus, $\mathop{\rm lim\ inf}\limits_{y\rightarrow x} Q_{\lambda}(y) = +\infty = Q_{\lambda}(x)$. We complete the proof immediately. 
\end{proof}

With the help of Lemma \ref{lower_semicontinuous}, we show the existence of optimal solutions to problems (\ref{l1/l2_minimization_problem_constraint}) and (\ref{l1_l2_minimization_problem_constraint_envelop}) in the next proposition.
\begin{proposition}
The optimal solution set of (\ref{l1_l2_minimization_problem_constraint_envelop}) is nonempty. Moreover, if the constraint set of problem (\ref{l1/l2_minimization_problem_constraint}) is nonempty, then the optimal solution set of problem (\ref{l1/l2_minimization_problem_constraint}) is nonempty.
\end{proposition}
\begin{proof}
It is obvious that problem (\ref{l1/l2_minimization_problem_constraint}) has optimal solutions when the constraint set is nonempty since it minimizes a continuous function over a bounded closed set. In order to show the existence of optimal solutions to problem (\ref{l1_l2_minimization_problem_constraint_envelop}), we first rewrite it to an equivalent formulation
$
{\rm min}\left\{Q_{\lambda}(x):x\in\mathbb{R}^{n}\right\}.
$
Clearly, $Q_{\lambda}$ is level bounded since ${\rm dom}(Q_{\lambda}) = D\backslash \{0_{n}\}$. Lemma \ref{lower_semicontinuous} ensures that $Q_{\lambda}$ is lower semicontinuous. Therefore, $Q_{\lambda}$ has minimizers. We then get the desired results.
\end{proof}

We next exhibit in the following theorem that optimal solutions to problem (\ref{l1_l2_minimization_problem_constraint_envelop}) tends to optimal solutions to problem (\ref{l1/l2_minimization_problem_constraint}) as $\lambda$ goes to zero. Since its proof is very similar with that of Theorem 17.1 in \cite{jorge2006numerical}, we omit it here.
\begin{theorem}
Let $\{\lambda_{k}>0:k\in\mathbb{N}\}$ be a decreasing sequence satisfying $\mathop{\rm lim}\limits_{k\rightarrow \infty}\lambda_{k}$$ = 0$. Suppose $x^{k}\in\mathbb{R}^{n}$ is an optimal solution of problem (\ref{l1_l2_minimization_problem_constraint_envelop}) with $\lambda = \lambda_{k}$ for $k\in\mathbb{N}$. There hold: \\
(i)\,$\{x^{k}:k\in\mathbb{N}\}$ is bounded;\\
(ii)\, Any accumulation point $x^{\star}$ of $\{x^{k}:k\in\mathbb{N}\}$ is a feasible point of problem (\ref{l1/l2_minimization_problem_constraint}), i.e., $x^{\star}\in D$ and $Ax^{\star}\in C$;\\
(iii)\,Any accumulation point $x^{\star}$ of $\{x^{k}:k\in\mathbb{N}\}$ is an optimal solution to problem (\ref{l1/l2_minimization_problem_constraint}).
\end{theorem}  

Recall that $x^{\star}$ is a (Fr{\'e}chet) stationary point of $F$ if and only if $0\in\hat{\partial} F(x^{\star})$. In the remaining part of this section, we dedicate to exploring the relationship between stationary points of problems (\ref{l1/l2_minimization_problem_constraint}) and (\ref{l1_l2_minimization_problem_constraint_envelop}). To this end, we first give a characterization of stationary points of problems (\ref{l1/l2_minimization_problem_constraint}) and (\ref{l1_l2_minimization_problem_constraint_envelop}) in the next lemma.
\begin{lemma}\label{charac_constraint_envelop}
The following statements hold if ${\rm int} D\cap {\rm int}\{x:Ax\in C\}\neq \emptyset$:\\
(i)\,The vector $x^{\star}\in \mathbb{R}^{n}$ is a stationary point of (\ref{l1/l2_minimization_problem_constraint}) if and only if $x^{\star}\in D$, $Ax^{\star}\in C$ and 
$
0\in \partial (\|\cdot\|_{1}+\iota_{D})(x^{\star}) - \frac{\|x^{\star}\|_{1}}{\|x^{\star}\|_{2}}\nabla\|x^{\star}\|_{2} + \partial (\iota_{C} \circ A )(x^{\star});
$

\noindent(ii)\,The vector $x^{\star}\in\mathbb{R}^{n}$ is a stationary point of (\ref{l1_l2_minimization_problem_constraint_envelop}) if and only if $0_{n}\neq x^{\star}\in D$ and
$
0\in \lambda\partial (\|\cdot\|_{1} + \iota_{D})(x^{\star})  - Q_{\lambda}(x^{\star})\nabla\|x^{\star}\|_{2} + \nabla ({\rm env}_{\iota_{C}}^{1}\circ A)(x^{\star}).
$
\end{lemma}
\begin{proof}
Noting that ${\rm env}_{\iota_{C}}^{1}\circ A$, $\iota_{C}\circ A$ and $\iota_{D}$ are all convex functions, all of them are regular functions. Due to the continuously differentiable property of $\|\cdot\|_{2}$ and ${\rm env}_{\iota_{C}}^{1}\circ A$ on $\mathbb{R}^{n}\backslash \{0_{n}\}$, we obtain from Corollary 1.111(i) in \cite{Boris1998Variational} that at any $x\in\mathbb{R}^{n}\backslash \{0_{n}\}$, $\frac{\|\cdot\|_{1}}{\|\cdot\|_{2}}$ and $\frac{\lambda\|\cdot\|_{1}+{\rm env}_{\iota_{C}}^{1}\circ A}{\|\cdot\|_{2}}$ are regular and there hold:
\begin{equation}\label{charac_constraint_envelop_a}
\hat{\partial}\frac{\|x\|_{1}}{\|x\|_{2}} = \partial\frac{\|x\|_{1}}{\|x\|_{2}} = \frac{\partial\|x\|_{1} - \frac{\|x\|_{1}}{\|x\|_{2}}\nabla\|x\|_{2}}{\|x\|_{2}},
\end{equation}
\begin{equation}\label{charac_constraint_envelop_b}
\begin{aligned}
\hat{\partial} \frac{\lambda\|x\|_{1}+{\rm env}_{\iota_{C}}^{1}(Ax)}{\|x\|_{2}}  &= \partial \frac{\lambda\|x\|_{1}+{\rm env}_{\iota_{C}}^{1}(Ax)}{\|x\|_{2}} \\
&= \frac{\lambda\partial\|x\|_{1}+\nabla ({\rm env}_{\iota_{C}}^{1}\circ A)x-\frac{\lambda\|x\|_{1}+{\rm env}_{\iota_{C}}^{1}(Ax)}{\|x\|_{2}}\nabla\|x\|_{2}}{\|x\|_{2}}.
\end{aligned}
\end{equation}

We first prove Item (i). The vector $x^{\star}\neq 0_{n}$ is a stationary point of problem (\ref{l1/l2_minimization_problem_constraint}) if and only if 
\begin{equation}\label{charac_constraint_envelop_c}
\begin{aligned}
0\in &\partial \frac{\|x^{\star}\|_{1}}{\|x^{\star}\|_{2}}+\partial(\iota_{D}+\iota_{C}\circ A)(x^{\star}),\\
= &\partial \frac{\|x^{\star}\|_{1}}{\|x^{\star}\|_{2}}+\partial \iota_{D}(x^{\star})+\partial (\iota_{C}\circ A)(x^{\star}),
\end{aligned}
\end{equation}
where the first inclusion follows from the locally Lipschitz continuity of $\frac{\|\cdot\|_{1}}{\|\cdot\|_{2}}$ at $x^{\star}$ as well as the regularity of $\frac{\|\cdot\|_{1}}{\|\cdot\|_{2}}$ and $\iota_{D}+\iota_{C}\circ A$ in $\mathbb{R}^{n}\backslash\{0_{n}\}$, the second equality holds thanks to Corollary 16.38 in \cite{Heinz2011Convex} and ${\rm int}D\cap {\rm int}\{x:Ax\in C\}\neq \emptyset$.
For $x\in\mathbb{R}^{n}\backslash\{0_{n}\}$, it follows from $\|x\|_{2}>0$ as well as $a\iota_{S}=\iota_{S}$ for any $a>0$ and $S\subseteq \mathbb{R}^{n}$ that
\begin{equation}\label{charac_constraint_envelop_e}
\|x\|_{2}\partial\iota_{D}(x) = \partial\iota_{D}(x) \ {\rm and}\ \|x\|_{2}\,\partial (\iota_{C}\circ A )(x) = \partial (\iota_{C}\circ A )(x).
\end{equation}
Therefore, we deduce from (\ref{charac_constraint_envelop_a}) and (\ref{charac_constraint_envelop_e}) that (\ref{charac_constraint_envelop_c}) is equivalent to 
$$
\begin{aligned}
0&\in \partial\|x^{\star}\|_{1} - \frac{\|x^{\star}\|_{1}}{\|x^{\star}\|_{2}}\nabla\|x^{\star}\|_{2} + \partial \iota_{D}(x^{\star}) + \partial (\iota_{C}\circ A) (x^{\star}),\\
&=\partial(\|\cdot\|_{1}+\iota_{D})(x^{\star}) - \frac{\|x^{\star}\|_{1}}{\|x^{\star}\|_{2}}\nabla\|x^{\star}\|_{2} + \partial (\iota_{C}\circ A) (x^{\star}),
\end{aligned}
$$
where the last equality follows from Corollary 16.38 in \cite{Heinz2011Convex}. This proves Item (i).

We next prove Item (ii) by a similar method. The vector $x^{\star}\in D\backslash\{0_{n}\}$ is a stationary point of problem (\ref{l1_l2_minimization_problem_constraint_envelop}) if and only if
\begin{equation}\label{charac_constraint_envelop_f}
0\in\partial\left(\frac{\lambda\|\cdot\|_{1}+{\rm env}_{\iota_{C}}^{1}\circ A}{\|\cdot\|_{2}}\right)(x^{\star})+\partial_{\iota_{D}}(x^{\star})
\end{equation}
from the regularity of $\frac{\lambda\|\cdot\|_{1}+{\rm env}_{\iota_{C}}^{1}\circ A}{\|\cdot\|_{2}}$ and $\ell_{D}$ as well as the locally Lipschitz continuity of $\frac{\lambda\|\cdot\|_{1}+{\rm env}_{\iota_{C}}^{1}\circ A}{\|\cdot\|_{2}}$. With the help of (\ref{charac_constraint_envelop_b}) and (\ref{charac_constraint_envelop_e}), (\ref{charac_constraint_envelop_f}) holds if and only if 
$$
\begin{aligned}
0\in & \lambda\partial \|x^{\star}\|_{1} + \nabla ({\rm env}_{\iota_{C}}^{1}\circ A)(x^{\star}) +\partial_{\iota_{D}}(x^{\star}) - \frac{\lambda\|x^{\star}\|_{1}+{\rm env}_{\iota_{C}}^{1}(Ax^{\star})}{\|x^{\star}\|_{2}}\nabla\|x^{\star}\|_{2}\\
= &\lambda\partial (\|\cdot\|_{1} + \iota_{D})(x^{\star}) + \nabla ({\rm env}_{\iota_{C}}^{1}\circ A)(x^{\star}) - Q_{\lambda}(x^{\star})\nabla\|x^{\star}\|_{2},
\end{aligned}
$$
where the last equality follows from $\lambda\iota_{D} = \iota_{D}$. We get Item (ii) immediately.
\end{proof}

We are now ready to reveal the relationship between stationary points of problems (\ref{l1/l2_minimization_problem_constraint}) and (\ref{l1_l2_minimization_problem_constraint_envelop}). Recall first that when $d>0$ and ${\rm int}\{x\in\mathbb{R}^{n}:Ax\in C\}\neq\emptyset$, there hold \cite{Chen2016Penalty}
\begin{equation}\label{recall_partial}
\begin{aligned}
\partial\iota_{D}(x) &= \left\{
\begin{aligned}
&\{0_{n}\}, \qquad  \qquad \quad  {\rm if}\ \|x\|_{2}<d, \\
&\{tx:t\ge 0\},  \qquad \, {\rm if}\ \|x\|_{2}=d,
\end{aligned}
\right.\\
\partial(\iota_{C}\circ A )(x) &= \left\{
\begin{aligned}
&\{0_{n}\}, \qquad  \qquad \qquad \qquad \quad \  {\rm if}\ \|Ax - b\|_{2}<\epsilon, \\
&\{tA^{T}(Ax-b):t\ge 0\}, \qquad  {\rm if}\ \|Ax - b\|_{2}=\epsilon.
\end{aligned}
\right.
\end{aligned}
\end{equation}
\begin{theorem}\label{relationship_stationary} 
Let $\{\lambda_{k}>0:k\in\mathbb{N}\}$ be a decreasing sequence satisfying $\mathop{\rm lim}\limits_{k\rightarrow \infty}\lambda_{k}$$ = 0$. Suppose $x^{k}\in\mathbb{R}^{n}$ is any stationary point of problem (\ref{l1_l2_minimization_problem_constraint_envelop}) with $\lambda = \lambda_{k}$ for $k\in\mathbb{N}$. Then, $\{x^{k}:k\in\mathbb{N}\}$ is bounded. Furthermore, if ${\rm int} D\cap {\rm int}\{x:Ax\in C\}\neq \emptyset$ and there exists a $\hat{x}\in\{x\in D:Ax\in C\}$ such that $\frac{Q_{\lambda_{k}}(x^{k})}{\lambda_{k}}\le \frac{\|\hat{x}\|_{1}}{\|\hat{x}\|_{2}}$ for all $k\in\mathbb{N}$, then there hold for any accumulation point $x^{\star}$ of $\{x^{k}:k\in\mathbb{N}\}$ that:\\
(i)\,$x^{\star}$ is a feasible point of problem (\ref{l1/l2_minimization_problem_constraint}), i.e., $x^{\star}\in D$ and $Ax^{\star}\in C$;\\
(ii)\,$x^{\star}$ is a stationary point of problem (\ref{l1/l2_minimization_problem_constraint}).
\end{theorem}
\begin{proof}
It is obvious that $\{x^{k}:k\in\mathbb{N}\}$ is bounded since $\|x^{k}\|_{2}\le d$ for all $k\in\mathbb{N}$.

We first prove Item (i). We have that 
$$
0\le\frac{{\rm env}_{\iota_{C}}^{1}(Ax^{k})}{\|x^{k}\|_{2}}\le Q_{\lambda_{k}}(x^{k})\le \frac{\lambda_{k}\|\hat{x}\|_{1}}{\|\hat{x}\|_{2}}.
$$
By the fact that $\mathop{\rm lim}\limits_{k\rightarrow \infty}\lambda_{k} = 0$, the above inequality implies that $\mathop{\rm lim}\limits_{k\rightarrow \infty}{\rm env}_{\iota_{C}}^{1}(Ax^{k}) = 0$. We then obtain $Ax^{\star}\in C$ from the definition and continuity of ${\rm env}_{\iota_{C}}^{1}$. We also have $x^{\star}\in D$ because $\|x^{k}\|_{2}\le d$ for $k\in\mathbb{N}$. Since $\|b\|_{2}>\epsilon$ and $Ax^{\star}\in C$, we have $x^{\star}\neq 0_{n}$.

We next prove Item (ii). We consider two different cases: $\|Ax^{\star}-b\|_{2}<\epsilon$ and $\|Ax^{\star}-b\|_{2} = \epsilon$, separately. 

Case 1. Suppose first that $x^{\star}$ satisfies $\|Ax^{\star}-b\|_{2}<\epsilon$. In this case $\partial(\iota_{C}\circ A)(x^{\star}) = \{0_{n}\}$. Then there exists $\{x^{k_{j}}:j\in\mathbb{N}\}$ and $J>0$ such that $\mathop{\rm lim}\limits_{j\rightarrow \infty}x^{k_{j}} = x^{\star}$ and $\|Ax^{k_{j}}-b\|_{2}<\epsilon$ for all $j>J$. By Lemma \ref{charac_constraint_envelop}, $x^{k_{j}}\in D\backslash \{0_{n}\}$ satisfies 
\begin{equation}\label{relationship_stationary_a}
0\in\lambda_{k_{j}}\partial (\|\cdot\|_{1}+\iota_{D})(x^{k_{j}}) - Q_{\lambda_{k_{j}}}(x^{k_{j}})\nabla\|x^{k_{j}}\|_{2} + \nabla({\rm env}_{\iota_{C}}^{1}\circ A)(x^{k_{j}}).
\end{equation}
Since $\|Ax^{k_{j}} - b\|_{2}<\epsilon$ for $j>J$, there hold $Q_{\lambda_{k_{j}}}(x^{k_{j}}) = \frac{\lambda_{k_{j}}\|x\|_{1}}{\|x\|_{2}}$ and $\nabla({\rm env}_{\iota_{C}}^{1}\circ A)(x^{k_{j}}) = 0_{n}$. Therefore, one has from (\ref{relationship_stationary_a}) that 
$
0\in\partial(\|\cdot\|_{1}+\iota_{D})(x^{k_{j}})- \frac{\|x^{k_{j}}\|_{1}}{\|x^{k_{j}}\|_{2}}\nabla\|x^{k_{j}}\|_{2}.
$
Taking limits on both sides of this inclusion, we deduce that 
$
0\in\partial(\|\cdot\|_{1}+\iota_{D})(x^{\star})-\frac{\|x^{\star}\|_{1}}{\|x^{\star}\|_{2}}\nabla\|x^{\star}\|_{2},
$
which indicates that $x^{\star}$ is a stationary point of problem (\ref{l1/l2_minimization_problem_constraint})  by Lemma \ref{charac_constraint_envelop} due to $\partial(\iota_{C}\circ A)(x^{\star}) = \{0_{n}\}$.

Case 2. Suppose now that $x^{\star}$ satisfies $\|Ax^{\star}-b\|_{2} = \epsilon$ and $\mathop{\rm lim}\limits_{j\rightarrow \infty}x^{k_{j}} = x^{\star}$. Lemma \ref{charac_constraint_envelop} leads to that
\begin{equation}\label{relationship_stationary_b}
\begin{aligned}
-\frac{1}{\lambda_{k_{j}}}&\left(\nabla({\rm env}_{\iota_{C}}^{1}\circ A)(x^{k_{j}}) - \frac{{\rm env}_{\iota_{C}}^{1}(Ax^{k_{j}})}{\|x^{k_{j}}\|_{2}}\nabla\|x^{k_{j}}\|_{2}\right)\\&\in\partial(\|\cdot\|_{1}+\iota_{D})(x^{k_{j}}) - \frac{\|x^{k_{j}}\|_{1}}{\|x^{k_{j}}\|_{2}}\nabla \|x^{k_{j}}\|_{2}.
\end{aligned}
\end{equation} 

We first prove $\left\{\tau_{j}:=\frac{(\|Ax^{k_{j}} - b\|_{2}-\epsilon)_{+}}{\lambda_{k_{j}}\|Ax^{k_{j}}-b\|_{2}}:j\in\mathbb{N}\right\}$ is bounded by contradiction. Assume that $\{\tau_{j}:j\in\mathbb{N}\}$ is unbounded. Without loss of generality, we suppose $\mathop{\rm lim}\limits_{j\rightarrow \infty}\tau_{j} = +\infty$. Then, it follows from ${\rm env}_{\iota_{C}}^{\lambda}$, $\nabla({\rm env}_{\iota_{C}}^{\lambda}\circ A)$ and (\ref{relationship_stationary_b}) that 
\begin{equation}\label{relationship_stationary_c}
\begin{aligned}
&-A^{T}(Ax^{k_{j}}-b) + \frac{\|Ax^{k_{j}}-b\|_{2}(\|Ax^{k_{j}}-b\|_{2}-\epsilon)_{+}}{2\|x^{k_{j}}\|_{2}}\nabla\|x^{k_{j}}\|_{2}\\
&\in \left(\partial(\|\cdot\|_{1}+\iota_{D})(x^{k_{j}}) - \frac{\|x^{k_{j}}\|_{1}}{\|x^{k_{j}}\|_{2}}\nabla\|x^{k_{j}}\|_{2}\right)/\tau_{j}.
\end{aligned}
\end{equation}
Since $\partial(\|\cdot\|_{1}+\iota_{D}) = \partial\|\cdot\|_{1}+\partial_{\iota_{D}}$, $\mathop{\cup}\limits_{j = 1}^{\infty}\partial\|x^{k_{j}}\|_{1}$ is bounded, $x^{\star}\neq 0_{n}$ and $\|Ax^{\star} - b\|_{2} = \epsilon$, taking limits on both sides of (\ref{relationship_stationary_c}) and invoking (\ref{recall_partial}), we obtain 
$
-A^{T}(Ax^{\star} - b) = t_{\star}x^{\star}
$
for some $t_{\star}\ge 0$. This implies that $x^{\star}$ satisfies 
\begin{equation}\label{relationship_stationary_d}
x^{\star}\in{\rm arg\,min}\left\{\eta(x) := \frac{1}{2}\|Ax-b\|_{2}^{2}+\frac{t_{\star}}{2}\|x\|_{2}^{2}:x\in\mathbb{R}^{n}\right\}.
\end{equation}
If $t_{\star} = 0$, (\ref{relationship_stationary_d}) contradicts to the fact that ${\rm int} D\cap {\rm int}\{x:Ax\in C\}\neq \emptyset$ due to $\|Ax^{\star} - b\|_{2}=\epsilon$. Thus, $t_{\star}$ should be larger than zero. In this case, we claim that $\|x^{\star}\|_{2} = d$. Otherwise, there exists $J^{'}>0$ such that for all $j>J^{'}$, $\|x^{k_{j}}\|_{2}<d$ and $\partial\iota_{D}(x^{k_{j}}) = \{0_{n}\}$. Then, the limitation of the righthand of (\ref{relationship_stationary_c}) is $\{0_{n}\}$, which implies that $t_{\star} = 0$ contradicting to $t_{\star} > 0$. Then, $\eta(x^{\star}) = \frac{1}{2}\epsilon^{2}+\frac{t^{\star}}{2}d^{2}$. Due to ${\rm int} D\cap {\rm int}\{x:Ax\in C\}\neq \emptyset$, for any $\widetilde{x}\in{\rm int} D\cap {\rm int}\{x:Ax\in C\}$, $\eta(\widetilde{x})<\eta(x^{\star})$, contradicting to (\ref{relationship_stationary_d}). We obtain $\{\tau_{j}:j\in\mathbb{N}\}$ is bounded immediately.

Without loss of generality, we assume $\mathop{\rm lim}\limits_{j\rightarrow \infty} \tau_{j} = \tau_{\star}$. It is obvious that 
$$
\mathop{\rm lim}\limits_{j\rightarrow \infty}\frac{{\rm env}_{\iota_{C}}^{1}(Ax^{k_{j}})}{\lambda_{k_{j}}\|x^{k_{j}}\|_{2}}\nabla\|x^{k_{j}}\|_{2} = \mathop{\rm lim}\limits_{j\rightarrow \infty} \frac{\tau_{j}\|Ax^{k_{j}}-b\|_{2}(\|Ax^{k_{j}}-b\|_{2}-\epsilon)_{+}}{2\|x^{k_{j}}\|_{2}}\nabla\|x^{k_{j}}\|_{2} = 0_{n}.
$$
Thus, we deduce by taking limits on both sides of (\ref{relationship_stationary_c}) that
$$
-\tau_{\star}A^{T}(Ax^{\star} - b)\in\partial(\|\cdot\|_{1}+\iota_{D})(x^{\star}) - \frac{\|x^{\star}\|_{1}}{\|x^{\star}\|_{2}}\nabla\|x^{\star}\|_{2},
$$
which implies that $x^{\star}$ is a stationary point of problem (\ref{l1/l2_minimization_problem_constraint}) by Lemma \ref{charac_constraint_envelop} from $\partial(\iota_{C}\circ A)(x^{\star}) = \{tA^{T}(Ax-b):t\ge 0\}$. We complete the proof immediately.
\end{proof}

We make several remarks on Theorem \ref{relationship_stationary}. First, we do not assume any constraint qualification condition here, although it is a common assumption when analyze the penalty methods. Second, the assumption ${\rm int} D\cap {\rm int}\{x: Ax\in C\}$ can be easily satisfied as long as $d$ is sufficiently large and ${\rm int}\{x: Ax\in C\}\neq \emptyset$. Finally, the assumption that there exists a $\hat{x}\in\{x\in D: Ax\in C\}$ such that $\frac{Q_{\lambda_{k}}(x^{k})}{\lambda_{k}}\le\frac{\|\hat{x}\|_{1}}{\|\hat{x}\|_{2}}$ for all $k\in\mathbb{N}$ can be satisfied by appropriately choosing the initial point of the algorithm for solving the penalty problem (\ref{l1_l2_minimization_problem_constraint_envelop}), if the algorithm generates nonincreasing objective sequence. For instance, let $\hat{x}\in\{x\in D: Ax\in C\}$. The initial point $x^{k,0}\in D$ of the algorithm for solving (\ref{l1_l2_minimization_problem_constraint_envelop}) with $\lambda = \lambda_{k}$ can be chosen as $x^{k,0} = \hat{x}$. Then, any accumulation point $x^{k}$ of the sequence generated by the algorithm satisfies $\frac{Q_{\lambda_{k}}(x^{k})}{\lambda_{k}}\le \frac{Q_{\lambda_{k}}(\hat{x})}{\lambda_{k}}=\frac{\|\hat{x}\|_{1}}{\|\hat{x}\|_{2}}$ due to $A\hat{x}\in C$.

\section{Parameterized proximal-gradient algorithms for solving problem (\ref{l1_l2_minimization_problem_constraint_envelop})}\label{S4}
In this section, we first study a general structured fractional optimization problem, of which problem (\ref{l1_l2_minimization_problem_constraint_envelop}) is a special case. We propose a parameterized proximal-gradient algorithm and its line search counterpart to solve the fractional optimization problem. Then, the proposed algorithm are applied to solving problem (\ref{l1_l2_minimization_problem_constraint_envelop}) and the closed-form solution of the involved proximity operator is given.
\subsection{A general fractional optimization problem}\label{4_1}
We first study in this subsection the following single-ratio fractional optimization problem
\begin{equation}\label{general_optimization_problem}
{\rm min}\left\{\frac{f(x)+h(x)}{g(x)}:x\in\Omega\cap {\rm dom}(f)\right\},
\end{equation}
where $f:\mathbb{R}^{n}\rightarrow \overline{\mathbb{R}}$ is proper lower semicontinuous, $g, h:\mathbb{R}^{n}\rightarrow \mathbb{R}$ and $\Omega:=\{x\in\mathbb{R}^{n}:g(x)\neq 0\}$. By defining $F:\mathbb{R}^{n}\rightarrow\overline{\mathbb{R}}$ at $x\in\mathbb{R}^{n}$ as 
$$
F(x):=\left\{
\begin{aligned}
&\frac{f(x)+h(x)}{g(x)},\qquad {\rm if}\ x\in\Omega\cap{\rm dom}(f),\\
&+\infty, \qquad \qquad \quad\,  {\rm otherwise},
\end{aligned}
\right.
$$
problem (\ref{general_optimization_problem}) can be written as 
\begin{equation}\label{general_optimization_problem_1}
{\rm min}\{F(x):x\in\mathbb{R}^{n}\}.
\end{equation}
We suppose in the whole sections \ref{S4} and \ref{S5}  
that $f, g$ and $h$ satisfy the following blanket assumptions.\\
Assumption 1.\\
(i)\,$f$ is locally Lipschitz continuous on ${\rm dom}(f)\cap \Omega$;\\
(ii)\,$g$ is locally Lipschitz continuously differentiable and positive on $\Omega\cap {\rm dom}(f)$;\\
(iii)\,$h$ is Lipschitz differentiable with a Lipschitz constant $L>0$;\\
(iv)\,$f+h$ is non-negative on ${\rm dom}(f)$, $\Omega\cap {\rm dom}(f)\neq \emptyset$;\\
(v)\,${\rm prox}_{f-\gamma g}(x)\neq \emptyset$ for any $x\in\mathbb{R}^{n}$ and $\gamma\ge0$;\\
(vi)\,$F$ is lower semicontinuous and level bounded.

One can check that problem (\ref{l1_l2_minimization_problem_constraint_envelop}) is a special case of problem (\ref{general_optimization_problem}) with $f = \lambda \|\cdot\|_{1} + \iota_{D}$, $h = {\rm env}_{\iota_{C}}^{1}\circ A$ and $g = \|\cdot\|_{2}$. It is clear that Items (i) and (ii) in Assumption 1 are satisfied in this case. Item (iii) in Assumption 1 holds since ${\rm env}_{\iota_{C}}^{1}\circ A$ is Lipschitz differentiable with a Lipschitz constant $\|A\|_{2}^{2}$ \cite{Na2017On}. Obviously, $f+h\ge 0$ and $\Omega\cap{\rm dom}(f) = \{x\in\mathbb{R}^{n}:0<\|x\|_{2}\le d\}\neq \emptyset$ as $d>0$ in this case. Because the function $f-\gamma g = \lambda\|\cdot\|_{1}-\gamma\|\cdot\|_{2}+\iota_{D}$ is proper, lower semicontinuous and bounded below on $\mathbb{R}^{n}$, ${\rm prox}_{f-\gamma g}(x)\neq \emptyset$ for any $x\in\mathbb{R}^{n}$ and $\gamma\ge0$. The lower semicontinuity and level boundedness of $F$ follow from Lemma \ref{lower_semicontinuous} and the boundedness of $D$ respectively. Thus, problem (\ref{l1_l2_minimization_problem_constraint_envelop}) satisfies Assumption 1. 
\subsection{The parameterized proximal-gradient algorithm (PPGA) for solving problem (\ref{general_optimization_problem})}
In order to solve the single-ratio fractional optimization problem (\ref{general_optimization_problem}), we first review the parametric programming for problem (\ref{general_optimization_problem}), which motivates us to develop the parameterized proximal-gradient algorithm.

The parametric approach, which may date back to Dinkelbach's algorithm \cite{Werner1967On}, applied to problem (\ref{l1_l2_minimization_problem_constraint_envelop}) generates the k-th iteration by solving a subproblem 
\begin{equation}\label{k-th_subproblem}
x^{k+1}\in{\rm arg\,min}\{f(x)-C_{k}g(x)+h(x):x\in\Omega\},
\end{equation}
where $C_{k}$ is updated by $C_{k}:=\frac{f(x^{k})+h(x^{k})}{g(x^{k})}$. In each iteration, one can apply proximal-gradient methods to the subproblem (\ref{k-th_subproblem}) to obtain a stationary point of problem (\ref{k-th_subproblem}). However, such algorithm may be not efficient enough since solving the nonconvex nonsmooth subproblem (\ref{k-th_subproblem}) in each iteration can yield high computational cost. Inspired by the above analysis, we intend to solve a much easier problem instead of problem (\ref{k-th_subproblem}). We propose to use a quadratic approximation for $h(x)$ and solve the following subproblem in each iteration
\begin{equation}\label{k-th_subproblem_first_order}
x^{k+1}\in{\rm arg\,min}\{f(x)-C_{k}g(x)+h(x^{k})+\langle\nabla h(x^{k}),x-x^{k}\rangle+\frac{\|x-x^{k}\|_{2}^{2}}{2\alpha_{k}}:x\in\Omega\},
\end{equation}
where $\alpha_{k}>0$ for $k\in\mathbb{N}$. With the help of the notion of proximity operators, (\ref{k-th_subproblem_first_order}) can be equivalently reformulated as 
$
x^{k+1}\in{\rm prox}_{\alpha_{k}(f-C_{k}g)}(x^{k}-\alpha_{k}\nabla h(x^{k})).
$
It is worth noting that solving (\ref{k-th_subproblem_first_order}) is much easier than solving problem (\ref{k-th_subproblem}) in many applications, since the proximity operator of $f-C_{k}g$ can be explicitly computed in these applications. We will show this point in subsection \ref{4_4}.

The proposed algorithm for solving problem (\ref{general_optimization_problem}) is summarized in Algorithm 1. Since Algorithm 1 involves in the proximity operator of $\alpha_{k}(f-C_{k}g)$, the gradient of $h$ and the parameter $C_{k}$, we refer to it as the parameterized proximal-gradient algorithm (PPGA). We remark here that the step size $\alpha_{k}$ in Algorithm 1 is required to be in $(\underline{\alpha},\overline{\alpha})\subset(0,1/L)$ to ensure the convergence, which will be established in section \ref{S5}.

\begin{algorithm}
\caption{parameterized proximal-gradient algorithm (PPGA) for solving problem (\ref{general_optimization_problem})}
\label{algorithm1}
\textbf{Step 0}: Input $x^{0}\in \Omega\cap{\rm dom}(f)$, $0<\underline{\alpha}\le \alpha_{k}\le\overline{\alpha}<1/L$ for $k\in \mathbb{N}$. Set $k\leftarrow 0$.\\
\textbf{Step 1}:  \quad Compute\\
\indent \qquad \qquad \quad $C_{k} = \frac{f(x^{k})+h(x^{k})}{g(x^{k})}$,\\
\indent \qquad \qquad \quad $x^{k+1}\in{\rm prox}_{\alpha_{k}(f-C_{k}g)}(x^{k}-\alpha_{k}\nabla h(x^{k}))$.\\
\textbf{Step 2}: Set $k\leftarrow k+1$ and go to Step 1.
\end{algorithm}

\subsection{PPGA with line search (PPGA$\_$L)}
In this subsection, we incorporate a line search scheme for adaptively choosing $\alpha_{k}$ into PPGA. In PPGA, the step size $\alpha_{k}$ should be less than $1/L$ for $k\in\mathbb{N}$ to ensure the convergence. However, the step size may be too small in the case of large $L$ and thus may lead to slow convergence of PPGA. To speed up the convergence, we take advantage of the line search technique in \cite{Wright2009Sparse} to enlarge the step size and meanwhile guarantee the convergence of the algorithm. The PPGA with line search scheme (PPGA$\_$L) is described in Algorithm \ref{Algorithm2}.  
\begin{algorithm}
\caption{PPGA with line search (PPGA$\_$L) for problem (\ref{general_optimization_problem})}
\label{Algorithm2}
\textbf{Step 0}: Input $x^{0}\in \Omega\cap{\rm dom}(f)$, $a\ge 0$, $0<\underline{\alpha}<\overline{\alpha}$, $0<\eta<1$,  and an integer $N\ge 0$. Set $k\leftarrow 0$.\\
\textbf{Step 1}:  $C_{k} = \frac{f(x^{k})+h(x^{k})}{g(x^{k})}$,\\
\indent\qquad\qquad choose $\alpha_{k,0}\in[\underline{\alpha},\overline{\alpha}]$\\
\textbf{Step 2}: For $m = 0,1,...,$ do\\
\indent\qquad\qquad\qquad $\alpha_{k} = \alpha_{k,0}\eta^{m}$,\\
\indent\qquad\qquad\qquad $\widetilde{x}^{k+1}\in {\rm prox}_{\alpha_{k}(f-C_{k}g)}(x^{k}-\alpha_{k}\nabla h(x^{k}))$,\\
\indent\qquad\qquad\qquad If $\widetilde{x}^{k+1}$ satisfies $\widetilde{x}^{k+1}\in \Omega\cap{\rm dom}(f)$ and \\
\indent\qquad\qquad(c) \qquad \ \ \ $F(\widetilde{x}^{k+1})\le \mathop{\rm max}\limits_{[k-N]_{+}\le j\le k}C_{j}-\frac{a}{2}\|\widetilde{x}^{k+1}-x^{k}\|_{2}^{2}$,\\
\indent\qquad\qquad\qquad Set $x^{k+1} = \widetilde{x}^{k+1}$ and go to Step 3.\\
\textbf{Step 3}: $k\leftarrow k+1$ and go to Step 1. 
\end{algorithm}

From the inequality (c) of Algorithm \ref{Algorithm2}, $\{F(x^{k}):k\in\mathbb{N}\}$ is monotone when $N=0$, while it is generally nonmonotone as $N>0$. For convenience of presentation, we refer to PPGA$\_$L as PPGA$\_$ML if $N=0$ and PPGA$\_$NL if $N>0$. Set $\Delta x:=x^{k}-x^{k-1}$, $\Delta h:= \nabla h(x^{k})- \nabla h(x^{k-1})$. Motivated from \cite{Barzilai1988Two,Lu2019Nonmonotone,Wright2009Sparse}, we usually choose $\alpha_{k,0}$ in the following formula
$$
\alpha_{k,0} = \left\{
\begin{aligned}
&{\rm max}\{\underline{\alpha},{\rm min}\{\overline{\alpha},\frac{\|\Delta x\|_{2}^{2}}{|\langle\Delta x,\Delta h\rangle|}\}\},\qquad {\rm if}\ \langle\Delta x,\Delta h\rangle\neq 0,\\
&\overline{\alpha}, \qquad \qquad \qquad \qquad \qquad \qquad \quad \ \ \, {\rm otherwise}.
\end{aligned}
\right.
$$ 
This choice of $\alpha_{k,0}$ can be viewed as an adaptive approximation of $1/L$ via some local curvature information of $h$. We will prove in the section \ref{5_4} that Step 2 in PPGA$\_$L terminates in finite steps and PPGA$\_$ML converges globally.
\subsection{Applications of PPGA and PPGA$\_$L to problem (\ref{l1_l2_minimization_problem_constraint_envelop})}\label{4_4}
We apply in this subsection PPGA and PPGA$\_$L to problem (\ref{l1_l2_minimization_problem_constraint_envelop}). As mentioned in subsection \ref{4_1}, 
problem (\ref{l1_l2_minimization_problem_constraint_envelop}) is a special case of problem (\ref{general_optimization_problem}) with $f = \lambda\|\cdot\|_{1}+\iota_{D}$, $h = {\rm env}_{\iota_{C}}^{1}\circ A$ and $g = \|\cdot\|_{2}$. Then, the Lipschitz constant $L$ of $h$ should be $\|A\|_{2}^{2}$ in this case. We next dedicate to showing that ${\rm prox}_{\alpha_{k}(f-C_{k}g)}$ can be efficiently computed in this case.

For $\gamma>0$, we define $\rho_{\gamma}:\mathbb{R}^{n}\rightarrow \overline{\mathbb{R}}$ as 
\begin{equation}\label{l1-l2-norm-constrained}
\rho_{\gamma} := \|\cdot\|_{1} - \gamma\|\cdot\|_{2} + \iota_{D}.
\end{equation}
Then, $\alpha_{k}(f-C_{k}g) = \alpha_{k}\lambda\rho_{\gamma}$ with $\gamma = C_{k}/\lambda$ for problem (\ref{l1_l2_minimization_problem_constraint_envelop}). We establish the closed-form solution of ${\rm prox}_{\beta\rho_{\gamma}}$ for $\beta,\gamma>0$ in the next proposition, which is inspired by Lemma 3.1 in \cite{Yifei2018Fast}. Before that, we first define the soft thresholding operator $\mathcal{S}:\mathbb{R}^{n}\times \mathbb{R}_{+}\rightarrow\mathbb{R}^{n}$ at $(y,\alpha)\in\mathbb{R}^{n}\times\mathbb{R}_{+}$ as
$$
\mathcal{S}(y,\alpha) := \left\{
\begin{aligned}
&y-\alpha, \qquad  {\rm if}\ y>\alpha,\\
&0, \qquad \qquad {\rm if}\ |y|\le \alpha,\\
&y+\alpha, \qquad \, {\rm otherwise}.
\end{aligned}
\right.
$$
\begin{proposition}
Let $y\in\mathbb{R}^{n}$, $\beta,\gamma>0$ and $\rho_{\gamma}$ be defined by (\ref{l1-l2-norm-constrained}). Then the following statements hold:\\
(i)If $\|y\|_{\infty}>\beta$,
$$
{\rm prox}_{\beta\rho_{\gamma}}(y) = \left\{
\begin{aligned}
&\frac{z(\|z\|_{2} + \beta\gamma)}{\|z\|_{2}}, \qquad  {\rm if} \ \|z\|_{2}\le d-\beta\gamma, \\
&\frac{zd}{\|z\|_{2}}, \qquad \qquad \quad \ \ {\rm otherwise},
\end{aligned}
\right.
$$ 
where $z = \mathcal{S}(y,\beta)$;\\
(ii) If $\|y\|_{\infty} = \beta$, $x^{\star}\in {\rm prox}_{\beta\rho_{\gamma}}(y)$ if and only if $\|x^{\star}\|_{2} = {\rm min}\{\beta\gamma,d\}$, $x_{i}^{\star} = 0$ for $i\in\{j\in\mathbb{N}_{n}:|y_{j}|<\beta\}$ and $x_{i}^{\star}y_{i}\ge 0$ for $i\in \{j\in \mathbb{N}_{n}: |y_{i}| = \beta\}$;
\\
(iii)If $(1-\gamma)\beta<\|y\|_{\infty}<\beta$, $x^{\star}\in {\rm prox}_{\beta\rho_{\gamma}}(y)$ if and only if $x^{\star}$ is a 1-sparse vector satisfying $\|x^{\star}\|_{2} = {\rm min}\{\|y\|_{\infty}+(\gamma-1)\beta,d\}$, $x_{i}^{\star}y_{i}\ge 0$ for $i\in\mathbb{N}_{n}$, $x_{i}^{\star} = 0$ for $i\in\{j\in\mathbb{N}_{n}:|y_{j}|<\|y\|_{\infty}\}$;\\
(iv)If $\|y\|_{\infty}\le (1-\gamma)\beta$, ${\rm prox}_{\beta\rho_{\gamma}}(y) = \{0_{n}\}$.
\end{proposition}
\begin{proof}
It is straightforward to obtain the following relations about the sign and order for the components of any $x^{\star}$ in ${\rm prox}_{\beta\rho_{\gamma}}(y)$:
\begin{equation}\label{prox_1}
x_{i}^{\star}\left\{
\begin{aligned}
&\ge 0, \qquad  {\rm if} \ y_{i}>0, \\
&\le 0, \qquad {\rm else},
\end{aligned}
\right.
\end{equation} 
and 
\begin{equation}\label{prox_8}
|x_{i}^{\star}|\ge |x_{j}^{\star}|\quad {\rm if}\ |y_{i}|\ge |y_{j}|.
\end{equation}
Otherwise, we can always change the sign of $x^{\star}$ or swap the absolute values of $x_{i}^{\star}$ and $x_{j}^{\star}$ to obtain a smaller objective. Therefore, we can assume without loss of generality that $y_{1}\ge y_{2}\ge \cdots \ge y_{n}\ge 0$.

Let $G:\mathbb{R}^{n}\rightarrow \mathbb{R}$ and $H:\mathbb{R}^{n}\rightarrow \mathbb{R}$ be defined at $x\in\mathbb{R}^{n}$ as $G(x):=\|x\|_{1} - \gamma\|x\|_{2} + \frac{1}{2\beta}\|x-y\|_{2}^{2}$ and $H(x):=d-\|x\|_{2}$, respectively. Then, the Lagrange function for the optimization problem 
$$
{\rm prox}_{\beta\rho_{\gamma}}(y) = {\rm arg\,min}\{\|x\|_{1}-\gamma\|x\|_{2}+\frac{1}{2\beta}\|x-y\|_{2}^{2}:\|x\|_{2}\le d\}
$$ 
has the form of 
$$
L(x,\eta) = G(x) - \eta H(x),
$$
where $\eta\ge 0$ is the Lagrange multipliers. Thus, any $x^{\star}\in{\rm prox}_{\beta\rho_{\gamma}}(y)$ must satisfy the following KKT condition for some $\eta^{\star}\ge 0$, due to the Lipschitz continuity of the objective,
\begin{equation}\label{prox_2}
\left\{
\begin{aligned}
&0\in\partial_{x}L(x^{\star},\eta^{\star}),\\
&\|x^{\star}\|_{2}\le d,\\
&\eta^{\star}\ge 0,\\
&\eta^{\star}(d-\|x^{\star}\|_{2}) = 0.
\end{aligned}
\right.
\end{equation}
Recall that $\partial\|\cdot\|_{2}(0_{n}) = \{\gamma\in\mathbb{R}^{n}:\|\gamma\|_{2}\le 1\}$ 
and $\partial(\|\cdot\|_{1} - \|\cdot\|_{2})\subseteq \partial\|\cdot\|_{1} - \partial\|\cdot\|_{2}$. Thus, the first relation in (\ref{prox_2}) implies that one of the following relations holds for some $p\in\partial\|\cdot\|_{1}(x^{\star})$:
\begin{equation}\label{prox_3}
(1-\frac{\beta\gamma-\eta^{\star}\beta}{\|x^{\star}\|_{2}})x^{\star} = y - \beta p \ {\rm and} \ x^{\star}\neq 0_{n},
\end{equation}
\begin{equation}\label{prox_4}
\|y-\beta p\|_{2}\le |\eta^{\star}\beta-\beta\gamma| \ {\rm and} \ x^{\star} = 0_{n}.
\end{equation}
One can easily check that for any $x^{\star}$ satisfying (\ref{prox_2}), there holds 
\begin{equation}\label{prox_5}
\begin{aligned}
G(x^{\star}) &= \|x^{\star}\|_{1} - \gamma\|x^{\star}\|_{2} + \frac{1}{2\beta}\|x^{\star}\|_{2}^{2} + \frac{1}{2\beta}\|y\|_{2}^{2} - \langle x^{\star},p+(\frac{1}{\beta}-\frac{\gamma-\eta^{\star}}{\|x^{\star}\|_{2}})x^{\star}\rangle\\
&=-\gamma\|x^{\star}\|_{2} + \frac{1}{2\beta}\|x^{\star}\|_{2}^{2} - (\frac{1}{\beta}-\frac{\gamma-\eta^{\star}}{\|x^{\star}\|_{2}})\|x^{\star}\|_{2}^{2} + \frac{1}{2\beta}\|y\|_{2}^{2}\\
&= -\frac{1}{2\beta}\|x^{\star}\|_{2}^{2} + \frac{1}{2\beta}\|y\|_{2}^{2} - \eta^{\star}\|x^{\star}\|_{2}\\
&\le G(0_{n}).
\end{aligned}
\end{equation}
Therefore, we need to find $x^{\star}$ with the largest norm among all $x^{\star}$ satisfying (\ref{prox_2}). In the following, we discuss the four cases listed in this proposition respectively.\\
(i) $y_{1}>\beta$. In this case, $y_{1}-\beta p_{1}>0$. Then (\ref{prox_1}) and (\ref{prox_3}) yield that $1-\frac{\beta\gamma-\eta^{\star}\beta}{\|x^{\star}\|_{2}}>0$. Thus, $x_{i}^{\star} = 0$ for $i\in\{j\in\mathbb{N}_{n}:y_{j}\le \beta\}$. Then, we deduce from (\ref{prox_3}) that $y-\beta p = \mathcal{S}(y,\beta)$ and 
\begin{equation}\label{prox_6}
\|x^{\star}\|_{2} = \|\mathcal{S}(y,\beta)\|_{2} + \beta\gamma -\eta^{\star}\beta.
\end{equation}
By the KKT condition (\ref{prox_2}) and (\ref{prox_6}), we have 
$$
\eta^{\star} = \left\{
\begin{aligned}
&0,\qquad \qquad \qquad \qquad \quad \ \ {\rm if} \ \|\mathcal{S}(y,\beta)\|_{2}+\beta\gamma\le d,\\
&\frac{\|\mathcal{S}(y,\beta)\|_{2}-d}{\beta} + \gamma, \qquad {\rm otherwise},
\end{aligned}
\right.
$$ 
$$
\|x^{\star}\|_{2} = \left\{
\begin{aligned}
&\|\mathcal{S}(y,\beta)\|_{2} + \beta\gamma, \qquad {\rm if} \|\mathcal{S}(y,\beta)\|_{2} + \beta\gamma\le d,\\
&d, \qquad \qquad \qquad  \qquad \ \, {\rm otherwise}.
\end{aligned}
\right.
$$
The above two relations and (\ref{prox_3}) imply the desired results.\\
(ii) $y_{1} = \beta$. We conclude from (\ref{prox_3}) that the left side of (\ref{prox_3}) is $0_{n}$. Thus, we have $x_{i}^{\star} = 0$ for $i\in\{j\in\mathbb{N}_{n}:y_{j}<\beta\}$ due to $p\in\partial \|\cdot\|_{1}(x^{\star})$ and $\|x^{\star}\|_{2} = \beta\gamma-\eta^{\star}\beta$. By the KKT condition and (\ref{prox_6}), $\|x^{\star}\|_{2} = {\rm min}\{d,\beta\gamma\}$ and $\eta^{\star} = \frac{\beta\gamma-{\rm min}\{d,\beta\gamma\}}{\beta}$. Then, any $x^{\star}\in{\rm prox}_{\beta\rho_{\gamma}}(y)$ should satisfy $\|x^{\star}\|_{2} = {\rm min}\{d,\beta\gamma\}$, $x_{i}^{\star}y_{i}\ge 0$ for $i\in\mathbb{N}_{n}$ and $x_{i}^{\star} = 0$ for $i\in\{j\in\mathbb{N}_{n}:y_{j}<\beta\}$. We next show that any $x^{\star}$ satisfying the above conditions belongs to ${\rm prox}_{\beta\rho_{\gamma}}(y)$, i.e., $G(x^{\star}) =$ const. One can check it by (\ref{prox_5}) and Item (ii) follows immediately.\\
(iii) $(1-\gamma)\beta<y_{1}<\beta$. One has $x_{1}^{\star}>0$ from (\ref{prox_1}) and (\ref{prox_8}). Then, it follows from (\ref{prox_3}) that $1-\frac{\beta\gamma-\eta^{\star}\beta}{\|x^{\star}\|_{2}}<0$, which leads to $\|x^{\star}\|_{2}<(\gamma-\eta^{\star})\beta$. We first show $x_{i}^{\star}=0$ for $i\in\{j\in\mathbb{N}_{n}:y_{j}<y_{1}\}$. Otherwise, there exist $x_{i}^{\star}>0$ and $y_{i}<y_{1}$. Then
$$
y_{1}-\beta p_{1} = y_{1}-\beta = (1-\frac{\beta\gamma-\eta^{\star}\beta}{\|x^{\star}\|_{2}})x_{1}^{\star} \le (1-\frac{\beta\gamma-\eta^{\star}\beta}{\|x^{\star}\|_{2}})x_{i}^{\star} = y_{i} - \beta p_{i} = y_{i} - \beta,
$$
which contradicts to $y_{i}<y_{1}$. We know from (\ref{prox_3}) that 
\begin{equation}\label{prox_7}
\|x^{\star}\|_{2} = (x-\eta^{\star})\beta-\|y-\beta p\|_{2}.
\end{equation} 
By (\ref{prox_2}) and (\ref{prox_6}), we should find $x^{\star}$ satisfying (\ref{prox_2}) with the largest norm. If $d\ge y_{1}+(\gamma-1)\beta$, we get from (\ref{prox_7}) that $x^{\star}$ should be chosen as a 1-sparse vector with $\|x^{\star}\|_{2} = \beta\gamma - (\beta-y_{1}) = y_{1} + (\gamma - 1)\beta\le d$. Else, $\|x^{\star}\|_{2}$ should be $d$. From (\ref{prox_5}), $\eta^{\star}$ should be chosen as the largest one among all $\eta^{\star}$ satisfying the KKT condition (\ref{prox_2}). Then, (\ref{prox_7}) encourages us to choose $x^{\star}$ to be a 1-sparse vector with $\|x^{\star}\|_{2} =  d$ and $\eta^{\star} = \frac{y_{1}+(\gamma-1)\beta-d}{\beta}>0$. We then obtain Item (iii) immediately.\\
(iv) $y_{1}\le (1-\gamma)\beta$. It is clear that (\ref{prox_3}) has no solution and (\ref{prox_4}) holds for $x^{\star} = 0_{n}$ and $\eta^{\star} = 0$. Then,  we get the desired result immediately.
\end{proof}

\section{Convergence analysis}\label{S5}
This section is devoted to the convergence analysis of PPGA and PPGA$\_$L for problem (\ref{general_optimization_problem}). To this end, we first propose a necessary and sufficient condition for stationary points of $F$, and then present a necessary condition based on proximity operators and gradients of related functions for stationary points of $F$. Next, we establish the convergence of objective function values and the subsequential convergence for PPGA. By assuming the KL property of the objective, we further prove the convergence of the whole sequence generated by PPGA. Then, we show subsequential convergence of PPGA$\_$NL and global sequential convergence of PPGA$\_$ML under the KL property assumption. Finally, by showing the objective function $Q_{\lambda}$ of problem (\ref{l1_l2_minimization_problem_constraint_envelop}) is semialgebraic, we obtain the global sequential convergence of PPGA and PPGA$\_$ ML applied to solving problem (\ref{l1_l2_minimization_problem_constraint_envelop}).
\subsection{Stationary points of $F$}
We study in this subsection the stationary points of $F$. A necessary and sufficient condition for stationary point of F based on Fr{\'e}chet subdifferential of $f$ and gradients of $g$ and $h$ is first proposed. Then, utilizing proximity operators and gradients of the related functions, we present a necessary condition for stationary point of $F$. This will help us proving any accumulation points of the sequence generated by PPGA or PPGA$\_$L is a stationary point of $F$.

We establish a necessary and sufficient condition for stationary points of $F$ in the next proposition.
\begin{proposition}\label{Directional stationary 1}
The vector $x^{\star}\in{\rm dom}(F)$ is a stationary point of $F$ if and only if
$
0\in \hat{\partial}f(x^{\star}) + \nabla h(x^{\star}) - C_{\star}\nabla g(x^{\star}),
$
where $C_{\star}:=F(x^{\star})$.
\end{proposition}
\begin{proof}
By Proposition \ref{P2.1}, we know that for $x\in{\rm dom}(F)$,
$$
\hat{\partial}F(x) = \frac{a_{2}(\hat{\partial}f(x)+\nabla h(x))-a_{1}\nabla g(x)}{a_{2}^{2}},
$$
where $a_{1} = f(x) + h(x)$ and $a_{2} = g(x)$. We then obtain this proposition immediately.
\end{proof}

We next derive a sufficient condition for stationary points of $F$ in the following proposition.
\begin{proposition}\label{sufficient_condition_directional_stationary}
If $x^{\star}\in{\rm dom}(F)$ satisfies
\begin{equation}\label{fixed-point-directional-stationary}
x^{\star}\in {\rm prox}_{\alpha(f-C_{\star}g)}(x^{\star}-\alpha\nabla h(x^{\star}))
\end{equation}
for some $\alpha>0$ and $C_{\star} = F(x^{\star})$, then $x^{\star}$ is a stationary point of $F$.
\end{proposition}
\begin{proof}
By the definition of proximity operators, (\ref{fixed-point-directional-stationary}) is equivalent to 
\begin{equation}\label{fixed-point-directional-stationary-1}
x^{\star}\in{\rm arg\,min}\{\alpha(f(x)-C_{\star}g(x))+\frac{1}{2}\|x-x^{\star}+\alpha\nabla h(x^{\star})\|_{2}^{2}:x\in\mathbb{R}^{n}\}.
\end{equation}
From the generalized Fermat's rule, (\ref{fixed-point-directional-stationary-1}) leads to
$
0\in\alpha\hat{\partial}f(x^{\star}) - \alpha C_{\star}\nabla g(x^{\star}) + \alpha \nabla h(x^{\star}),
$
which implies that $x^{\star}$ is a stationary point of $F$ by proposition \ref{Directional stationary 1}.
\end{proof}

\subsection{Subsequential convergence of PPGA}\label{5_2}
We prove in this subsection that the sequence $\{x^{k}:k\in\mathbb{N}\}$ generated by PPGA is bounded and any accumulation point of it is a stationary point of $F$. To this end, we first illustrate in the next theorem that $x^{k}\in{\rm dom}(F)$ for $k\in\mathbb{N}$ and $\{F(x^{k}):k\in\mathbb{N}\}$ is decreasing and convergent.
\begin{theorem}\label{convergence_al3}
The sequence $\{x^{k}:k\in\mathbb{N}\}$ generated by PPGA falls into ${\rm dom}(F)$ and the following statements holds:\\
(i)$F(x^{k+1})+\frac{1/\alpha_{k}-L}{2g(x^{k+1})}\|x^{k+1}-x^{k}\|_{2}^{2}\le F(x^{k}) \quad {\rm for}\ k\in\mathbb{N}$;\\
(ii)$\mathop{\rm lim}\limits_{k\rightarrow \infty}C_{k} = \mathop{\rm lim}\limits_{k\rightarrow \infty}F(x^{k}) = C\quad {\rm with}\ C\ge 0$;\\
(iii)$\mathop{\rm lim}\limits_{k\rightarrow \infty}\frac{1/\alpha_{k}-L}{g(x^{k+1})}\|x^{k+1}-x^{k}\|_{2}^{2} = 0$.
\end{theorem}
\begin{proof}
We first prove $x^{k}\in{\rm dom}(F)$ for $k\in\mathbb{N}$ by induction. First, the initial point $x^{0}\in{\rm dom}(F)$. Suppose $x^{k}\in{\rm dom}(F)$ for some $k\in\mathbb{N}$. From PPGA and the definition of proximity operators, we have 
$$
\begin{aligned}
f(x^{k+1}) &- C_{k}g(x^{k+1}) + \frac{1}{2\alpha_{k}}\|x^{k+1}-x^{k}+\alpha_{k}\nabla h(x^{k})\|_{2}^{2}\\
&\le f(x^{k}) - C_{k}g(x^{k}) +\frac{1}{2\alpha_{k}}\|\alpha_{k}\nabla h(x^{k})\|_{2}^{2},
\end{aligned}
$$
which together with $C_{k} = \frac{f(x^{k})+h(x^{k})}{g(x^{k})}$ implies that
\begin{equation}\label{convergence_al3_1}
f(x^{k+1}) - C_{k}g(x^{k+1}) + \frac{1}{2\alpha_{k}}\|x^{k+1}-x^{k}\|_{2}^{2} + \langle x^{k+1} - x^{k},\nabla h(x^{k})\rangle\le -h(x^{k}).
\end{equation}
From Item (iii) of Assumption 1, there holds
\begin{equation}\label{convergence_al3_2}
h(x^{k+1})\le h(x^{k}) + \langle\nabla h(x^{k}), x^{k+1}-x^{k}\rangle + \frac{L}{2}\|x^{k+1}-x^{k}\|_{2}^{2}.
\end{equation}
Summing (\ref{convergence_al3_1}) and (\ref{convergence_al3_2}) leads to 
\begin{equation}\label{convergence_al3_3}
f(x^{k+1}) + h(x^{k+1}) + \frac{1/\alpha_{k}-L}{2}\|x^{k+1} - x^{k}\|_{2}^{2} \le C_{k}g(x^{k+1}).
\end{equation}
Assume $x^{k+1}\notin {\rm dom}(F)$. We know that $x^{k+1}\notin \Omega$ and $g(x^{k+1}) = 0$ due to $x^{k+1}\in {\rm dom}(f-C_{k}g) = {\rm dom}(f)$ and ${\rm dom}(F) = {\rm dom}(f) \cap \Omega$. By Item (iv) of Assumption 1 and $0<\alpha_{k}<1/L$, we deduce that $x^{k+1} = x^{k}$ from (\ref{convergence_al3_3}). This contradicts to $x^{k}\in{\rm dom}(F)$ and thus implies $x^{k+1}\in{\rm dom}(F)$. Therefore, we conclude that $x^{k}\in{\rm dom}(F)$ for all $k\in\mathbb{N}$.

By dividing $g(x^{k+1})$ at both sides of (\ref{convergence_al3_3}), we obtain Item (i) immediately. Item (ii) follows immediately by $F\ge 0$ and $0<\alpha_{k}<1/L$. Item (iii) is a direct consequence of Items (i) and (ii). We complete the proof.
\end{proof}

We are now ready to present the main result of this subsection.
\begin{theorem}\label{eeee}
Let $\{x^{k}:k\in\mathbb{N}\}$ be generated by PPGA. Then $\{x^{k}:\mathbb{N}\}$ is bounded and any accumulation point of it is a stationary point of $F$.
\end{theorem}
\begin{proof}
The boundedness of $\{x^{k}:k\in\mathbb{N}\}$ follows from the level boundedness of $F$ and Theorem \ref{convergence_al3} (i). Let $\{x^{k_{j}}:j\in\mathbb{N}\}$ be a subsequence such that $\mathop{\rm lim}\limits_{j\rightarrow \infty}x^{k_{j}} = x^{\star}$. From Theorem \ref{convergence_al3} (i) and the lower semicontinuity of $F$, we deduce that $F(x^{\star})\le \mathop{\rm lim}\limits_{j\rightarrow \infty}F(x^{k_{j}})\le F(x^{0})$, which implies that $x^{\star}\in{\rm dom}(F)$. Thus, $g(x^{\star})\neq 0$ and $x^{\star}\in{\rm dom}(f)$. By Theorem \ref{convergence_al3} and $\alpha_{k}\le \overline{\alpha}$, we obtain 
$
F(x^{k_{j}}) + \frac{1/\overline{\alpha}-L}{2g(x^{k_{j}})}\|x^{k_{j}} - x^{k_{j}-1}\|_{2}^{2} \le F(x^{k_{j}-1}).
$
Using Theorem \ref{convergence_al3} (ii), $\overline{\alpha}<1/L$ and the continuity of $g$ at $x^{\star}$, we deduce that $\mathop{\rm lim}\limits_{j\rightarrow \infty}\|x^{k_{j}} - x^{k_{j}-1}\|_{2} = 0$ and $\mathop{\rm lim}\limits_{j\rightarrow \infty}x^{k_{j}-1} = x^{\star}$. Without loss of generality, we assume $\mathop{\rm lim}\limits_{j\rightarrow \infty}\alpha_{k_{j}-1} = \alpha$ with $0<\underline{\alpha}\le \alpha \le \overline{\alpha}$. From PPGA, we get
$$
x^{k_{j}}\in {\rm prox}_{\alpha_{k_{j}-1}(f-C_{k_{j}-1}g)}(x^{k_{j}-1}-\alpha_{k_{j}-1}\nabla h(x^{k_{j}-1})).
$$
As $f, g, h$ and $\nabla h$ are continuous on ${\rm dom}(F)$, we obtain (\ref{fixed-point-directional-stationary}) by passing to the limit in the above relation. Finally, by Proposition \ref{sufficient_condition_directional_stationary} we have that $x^{\star}$ is a stationary point of $F$.
\end{proof}
\subsection{Global sequential convergence of PPGA}\label{5_3}
In this subsection, we investigate the global convergence of the entire sequence $\{x^{k}:k\in\mathbb{N}\}$ generated by PPGA. We shall show $\{x^{k}:k\in\mathbb{N}\}$ converges to a stationary point of $F$ under the KL property assumption. Our analysis in this subsection mainly takes advantage of Proposition \ref{KL_property} which is based on KL property. If $F$ satisfies the KL property, from Proposition \ref{KL_property} and Theorem \ref{eeee}, we can establish the global convergence of PPGA by showing the boundedness of the generated sequence and Items (i)-(iii) of Proposition \ref{KL_property}. The boundedness of $\{x^{k}:k\in\mathbb{N}\}$ is obtained in Theorem \ref{eeee} and Item (iii) of proposition \ref{KL_property} is a direct consequence of the continuity of $F$. We next prove Items (i)-(ii) in the following lemma. 
\begin{lemma}\label{L5_5}
Let $\{x^{k}:k\in\mathbb{N}\}$ be generated by PPGA. Then the following statements hold:\\
(i)There exists $a>0$ such that
$
F(x^{k+1}) + \frac{a}{2}\|x^{k+1}-x^{k}\|_{2}^{2}\le F(x^{k})
$ 
for all $k\in\mathbb{N}$;\\
(ii)There exist $b>0$ and $w^{k+1}\in\partial F(x^{k+1})$ such that
$
\|w^{k+1}\|_{2}\le b\|x^{k+1}-x^{k}\|_{2}
$
for all $k\in\mathbb{N}$.
\end{lemma}
\begin{proof}
We first prove Item (i). Let $a:=(1/\overline{\alpha}-L)/M$ with $M:={\rm sup}\{g(x):F(x)\le F(x^{0})\}$. Since $g$ is continuous as well as $F$ is level bounded and lower semicontinuous, $M$ is finite. Thus $a>0$. This together with Theorem \ref{convergence_al3} (i) and $\alpha_{k}\le\overline{\alpha}$ yields Item (i).

We next prove Item (ii). By Theorem \ref{convergence_al3} and Theorem \ref{eeee}, we know that $x^{k}\in\Omega$ for $k\in\mathbb{N}$ and any accumulation point $x^{\star}$ of $\{x^{k}:k\in\mathbb{N}\}$ satisfies $g(x^{\star})>0$. Therefore, there exists $t>0$ such that $g(x^{k})\ge t$ for $k\in\mathbb{N}$, since $\{x^{k}:k\in\mathbb{N}\}$ is bounded and $g$ is continuous on $\Omega$. Let $S$ be the closure of $\{x^{k}:k\in\mathbb{N}\}$. By Theorem \ref{eeee}, $S$ is bounded and $S\subseteq {\rm dom}(F)$. Then, it can be easily checked that $F$ are globally Lipschitz continuous on $S$, with the help of Assumption 1 (i)-(iii). We denote the Lipschitz constant of $F$ by $\widetilde{L}$.

From PPGA,  the definition of proximity operators and the generalized Fermat's rule, we obtain that 
$$
x^{k}-x^{k+1}-\alpha_{k}\nabla h(x^{k})\in \alpha_{k}\hat{\partial}f(x^{k+1})-\alpha_{k}C_{k}\nabla g(x^{k+1}),
$$
which implies that
$$
\frac{x^{k}-x^{k+1}}{\alpha_{k}g(x^{k+1})} - \frac{\nabla h(x^{k})}{g(x^{k+1})} + \frac{C_{k}\nabla g(x^{k+1})}{g(x^{k+1})} \in \frac{\hat{\partial}f(x^{k+1})}{g(x^{k+1})}.
$$
Since $\hat{\partial}F = \frac{g(\hat{\partial}f + \nabla h)-(f+h)\nabla g}{g^{2}}$ on ${\rm dom}(F)$ by Proposition \ref{P2.1}, we have $w^{k+1}\in\hat{\partial}F(x^{k+1})$ with 
$$
w^{k+1}:=\frac{x^{k}-x^{k+1}}{\alpha_{k}g(x^{k+1})}+\frac{\nabla h(x^{k+1})-\nabla h(x^{k})}{g(x^{k+1})} + \frac{(C_{k}-C_{k+1})\nabla g(x^{k+1})}{g(x^{k+1})}.
$$
A direct computation yields that
$$
\|w^{k+1}\|_{2} \le (\frac{1}{\alpha_{k}t}+\frac{L}{t}+\frac{\widetilde{L}\|\nabla g(x^{k+1})\|_{2}}{t})\|x^{k+1}-x^{k}\|_{2}.
$$
Since $\{x^{k}:k\in\mathbb{N}\}$ is bounded and $\nabla g$ is continuous on $\Omega$, there exist $\beta>0$ such that $\|\nabla g(x^{k+1})\|_{2}\le \beta$ for $k\in\mathbb{N}$. Due to $\alpha_{k}\ge \underline{\alpha}>0$, we finally obtain that $\|w^{k+1}\|_{2}\le b\|x^{k+1}-x^{k}\|_{2}$ for $k\in\mathbb{N}$, where $b:=(1/\underline{\alpha}+L+\tilde{L}\beta)/t$. We complete the proof from $\hat{\partial}F(x^{k+1})\subseteq\partial F(x^{k+1})$.
\end{proof}

The main result of this subsection is established in the following theorem.
\begin{theorem}\label{converge_PPGA}
Let $\{x^{k}:k\in\mathbb{N}\}$ be generated by PPGA. If $F$ satisfies the KL property at any point in ${\rm dom}(F)$, then $\sum_{k=1}^{\infty}\|x^{k}-x^{k-1}\|_{2}< +\infty$ and $\{x^{k}:k\in\mathbb{N}\}$ converges to a stationary point of $F$.
\end{theorem}
\begin{proof}
From Theorem \ref{eeee}, it suffices to show that $\sum_{k=1}^{\infty}\|x^{k}-x^{k-1}\|_{2}\le+\infty$ and $\{x^{k}:k\in\mathbb{N}\}$ is convergent. According to Proposition \ref{KL_property}, we obtain this theorem immediately by utilizing Theorem \ref{eeee} and Lemma \ref{L5_5}.
\end{proof}
\subsection{Convergence analysis of PPGA$\_$L}\label{5_4}
This subsection is devoted to the convergence of PPGA$\_$L. We show that any accumulation point of $\{x^{k}:k\in \mathbb{N}\}$ generated by PPGA$\_$L is a stationary point of $F$. Furthermore, the global convergence of the entire sequence generated by PPGA$\_$ML is established under the KL property assumption. Because the analysis in this subsection is similar with that of \cite{Lu2019Nonmonotone, Wright2009Sparse,Na2020First}, we only present the main result in the next theorem and omit its proof due to the limitation of space.
\begin{theorem}\label{converge_PPGAL}
The following statements hold:\\
(i)Step 2 of PPGA$\_$L terminates at some $\alpha_{k}\ge \widetilde{\alpha}$ in most $T$ iterations, where
$$
\widetilde{\alpha}:=\eta/(aM+L),\quad M:={\rm sup}\{g(x):F(x)\le C_{0}\},
$$
$$
T:=\lceil\frac{-{\rm log}(\widetilde{\alpha}(aM+L))}{{\rm log}\eta} + 1\rceil;
$$
(ii)Let $\{x^{k}:k\in\mathbb{N}\}$ be generated by PPGA$\_$L. Then $\{x^{k}:k\in\mathbb{N}\}$ is bounded and any accumulation point of it is a stationary point of $F$;\\
(iii)Let $\{x^{k}:k\in\mathbb{N}\}$ be generated by PPGA$\_$ML, i.e., PPGA$\_$L when $N=0$. Then, $\sum_{k = 1}^{\infty}\|x^{k}-x^{k-1}\|_{2}<+\infty$ and $\{x^{k}:k\in\mathbb{N}\}$ converges to a stationary point of $F$, if $F$ satisfies the KL property at any point in ${\rm dom}(F)$.
\end{theorem}
\subsection{Convergence of PPGA and PPGA$\_$L applied to problem (\ref{l1_l2_minimization_problem_constraint_envelop})}
In this section, we establish the convergence of PPGA and PPGA$\_$L when they are applied to solving problem (\ref{l1_l2_minimization_problem_constraint_envelop}). As discussed in subsection \ref{5_3} and \ref{5_4}, the global sequential convergence of PPGA and PPGA$\_$ML relies on the KL property of the objective. The next lemma presents that the objective function $Q_{\lambda}$ of problem (\ref{l1_l2_minimization_problem_constraint_envelop}) is a semialgebraic function and thus satisfies the KL property. For the definition of the semialgebraic function and its relation to the KL property, we refer readers to \cite{Hedy2013Convergence}.
\begin{lemma}\label{Q_semialgebraic}
For any $\lambda>0$, $Q_{\lambda}$ is a semialgebraic function.
\end{lemma}  
\begin{proof}
According to the stability properties of semialgebraic functions \cite{Hedy2013Convergence}, we know that finite sums and products of semialgebraic functions are semialgebraic. Therefore, it suffices to show $\lambda\|\cdot\|_{1}$, ${\rm env}_{\iota_{C}}^{1}\circ A$, $\frac{1}{\|\cdot\|_{2}}$, and the indicator function on $\{x\in\mathbb{R}^{n}:\|x\|_{2}\le d, x\neq 0_{n}\}$ are semialgebraic. It has been shown in \cite{Hedy2013Convergence} that $\lambda\|\cdot\|_{1}$ is semialgebraic. Then we prove ${\rm env}_{\iota_{C}}^{1}\circ A$ and $\frac{1}{\|\cdot\|_{2}}$ are semialgebraic, i.e., Graph(${\rm env}_{\iota_{C}}^{1}\circ A$) and Graph($\frac{1}{\|\cdot\|_{2}}$) are semialgebraic sets. We have
$$
\begin{aligned}
{\rm Graph}({\rm env}_{\iota_{C}}^{1}\circ A) &= \{(x,s)\in\mathbb{R}^{n}\times \mathbb{R}:\frac{1}{2}(\|Ax-b\|_{2}-\epsilon)_{+}^{2} = s\}\\
=&\{(x,s)\in\mathbb{R}^{n}\times \mathbb{R}:\|Ax-b\|_{2}\le\epsilon, s=0\}\cup \\
&\{(x,s)\in\mathbb{R}^{n}\times \mathbb{R}:-\|Ax-b\|_{2}\\
\le& -\epsilon, \|Ax-b\|_{2}^{4}-(2\epsilon^{2}+4s)\|Ax-b\|_{2}^{2}+4s^{2}-4\epsilon^{2}s+\epsilon^{4} = 0\},\\
{\rm Graph}(\frac{1}{\|\cdot\|_{2}}) =&\{(x,s)\in\mathbb{R}^{n}\times \mathbb{R}:\frac{1}{\|x\|_{2}} = s\}\\
=&\{(x,s)\in\mathbb{R}^{n}\times \mathbb{R}: s^{2}\|x\|_{2}^{2} = 1\}.
\end{aligned}
$$
The above equalities imply that both ${\rm env}_{\iota_{C}}^{1}\circ A$ and $\frac{1}{\|\cdot\|_{2}}$ are semialgebraic. We next prove the indicator function on $\{x\in\mathbb{R}:\|x\|_{2}\le d,x\neq 0_{n}\}$ is semialgebraic, which amounts to showing $\{x\in\mathbb{R}:\|x\|_{2}\le d,x\neq 0_{n}\}$ is a semialgebraic set. It is clear that
$$
\{x\in\mathbb{R}^{n}:\|x\|_{2}\le d, x\neq 0_{n}\} = \{x\in\mathbb{R}^{n}:\|x\|_{2}\le d, -\|x\|_{2}^{2}<0\},
$$
which indicates that this set is semialgebraic. Then, we complete the proof immediately.
\end{proof}

With the help of Lemma \ref{Q_semialgebraic}, Theorems \ref{converge_PPGA} and \ref{converge_PPGAL}, we obtain the main result of this subsection in the following theorem.
\begin{theorem}
The following statements hold:\\
(i) If $\{x^{k}:x\in\mathbb{N}\}$ is generated by PPGA or PPGA$\_$ML applied to problem (\ref{l1_l2_minimization_problem_constraint_envelop}), then $\sum_{k = 1}^{\infty}\|x^{k}-x^{k-1}\|_{2}<+\infty$ and $\{x^{k}:k\in\mathbb{N}\}$ converges to a stationary point of $Q_{\lambda}$. \\
(ii) If $\{x^{k}:x\in\mathbb{N}\}$ is generated by PPGA$\_$NL, then $\{x^{k}:x\in\mathbb{N}\}$ is bounded and accumulation point of it is a stationary point of $Q_{\lambda}$.
\end{theorem}
\begin{proof}
In view of Lemma \ref{Q_semialgebraic}, $Q_{\lambda}$ satisfies the KL property at any point in ${\rm dom}(F)$. We then obtain this theorem by Theorems \ref{converge_PPGA} and \ref{converge_PPGAL}.
\end{proof}

\section{Numerical experiments}\label{S6}
In this section, we conduct some numerical simulations on sparse signal recovery to test the efficiency of our proposed algorithms, namely, PPGA, PPGA$\_$ML, and PPGA$\_$NL. All the experiments are carried out on a MAC with an Intel Core i5 CPU (2.7GHz) and Matlab 2019b.

We consider sparse signal recovery problems on two types of matrices: random Gaussian and random oversampled discrete cosine transform (DCT) matrices. A random Gaussian matrix has i.i.d standard Gaussian entries, while a random oversampled DCT matrix is specified as $A = [a_{1},...,a_{n}]\in\mathbb{R}^{m\times n}$ with 
$$
a_{j} = \frac{1}{\sqrt{m}}{\rm cos}\left(\frac{2\pi w j}{F}\right), \quad j = 1,2,...,n,
$$
where $w\in\mathbb{R}^{m}$ has i.i.d entries uniformly chosen in $[0,1]$, $F>0$ is a parameter to control the coherence in a way that a larger $F$ yields a more coherent matrix. In order to generate a $s$-sparse ground truth signal $x_{g}\in\mathbb{R}^{n}$ with at most $s$ nonzero entries, we first randomly choose a support set $J\subset \{1,2,...,n\}$ with $|J| = s$. Then, following the work of \cite{Chao2019Accelerated,Zeng2021Analysis}, when the sensing matrix is an oversampled DCT matrix, the nonzero elements are generated by the following MATLAB command: 
\begin{equation}\label{groundtruth}
{\rm xg = sign(randn(s,1)).*10.\, \hat{ }\, (D*rand(s,1))},
\end{equation}
where $randn$ and $rand$ are the MATLAB commands for the standard Gaussian distribution and the uniform distribution respectively, D is an exponential factor that controls the dynamic range $\Theta(x) = \frac{{\rm max}\{|x_{g}|\}}{{\rm min}\{|x_{g}|\}}$ of $x_{g}$. Following the work of \cite{Yifei2018Fast}, when the sensing matrix is a random Gaussian matrix, $x_{g}(J)$ has i.i.d standard Gaussian entries. The initial point $x^{0}$ of all the compared algorithms throughout this section except \ref{Algorithmic comparison} is chosen as an approximate solution of the $\ell_{1}$ problem, which is obtained by applying 2$n$ ADMM iterations to the problem 
\begin{equation}\label{initial_solution}
{\rm min}\{0.08\|x\|_{1}+\frac{1}{2}\|Ax-b\|_{2}^{2}:x\in\mathbb{R}^{n}\}.
\end{equation}
The stopping criterions for all the compared algorithms are the relative error $\|x^{k+1}-x^{k}\|_{2}/\|x^{k}\|_{2} \le 10^{-8}$ or $k>500n$.
\subsection{Comparison on PPGA, PPGA$\_$ML and PPGA$\_$NL}\label{6_1}
In this section, we compare the computational efficiency of our proposed algorithms, namely, PPGA, PPGA$\_$ML and PPGA$\_$NL. We consider the noise-free sparse signal recovery problem on random oversampled DCT matrices with $m = 64, n = 1024, s = 5, F\in \{1,5\}$ and $D\in\{1,2,3\}$. In this case, $\epsilon = 0$ and we set the parameter $d = 10^{7}$, $\lambda = 0.008$ for problem (\ref{l1_l2_minimization_problem_constraint_envelop}). For PPGA $\alpha_{k} = \frac{0.999}{\|A\|_{2}^{2}}$, for PPGA$\_$ML and PPGA$\_$NL $\eta = 0.5$, $a = 10^{-8}$ and $N = 4$.

Figure \ref{computation-efficient} plots the objective function values of problem (\ref{l1_l2_minimization_problem_constraint_envelop}), that is $Q_{\lambda}(x^{k})$, versus iteration numbers $k\in\mathbb{N}$ for PPGA, PPGA$\_$ML and PPGA$\_$NL. We find that in all the cases, $Q_{\lambda}(x^{k})$ generated by PPGA and PPGA$\_$ML decreases as $k$ increases. This verifies Theorem \ref{convergence_al3} (i) and Algorithm 2 Step 2 (c) when $N = 0$. It can be observed from Figure \ref{computation-efficient} that objective function values of PPGA$\_$ML and PPGA$\_$NL decrease much faster than those of PPGA. For some cases, PPGA$\_$ML and PPGA$\_$NL performs very similarly, see the top four plots of Figure \ref{computation-efficient}. While PPGA$\_$ML (resp., PPGA$\_$NL) performs slightly better in terms of computational efficiency than PPGA$\_$NL (resp., PPGA$\_$ML) for the bottom right plot (resp., bottom left plot) of Figure \ref{computation-efficient}. We also find from Figure \ref{computation-efficient} that as the difficulty of the sparse signal recovery problem is enhanced, that is F and D increase, all the compared algorithms need more iterations to converge.

\begin{figure}[!htbp]
\centering
\includegraphics[width=0.32\linewidth]{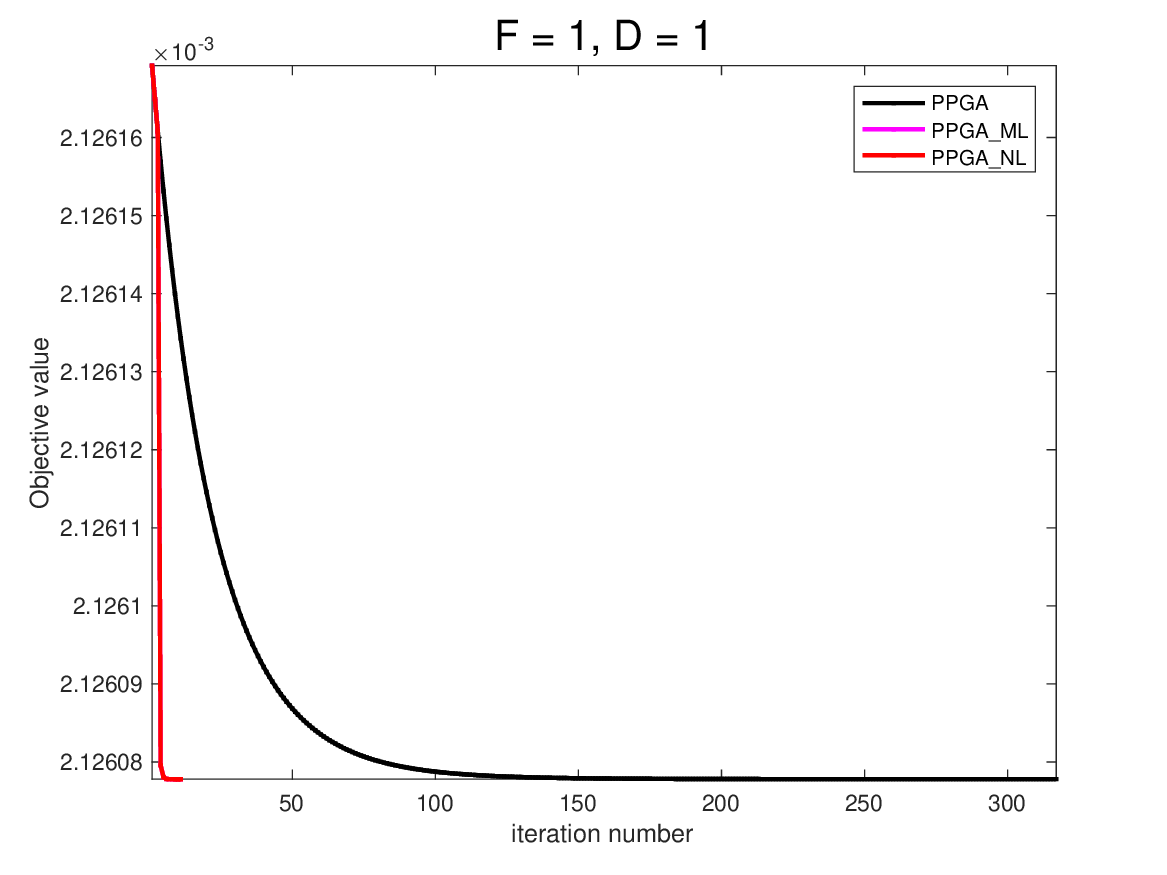}
\includegraphics[width=0.32\linewidth]{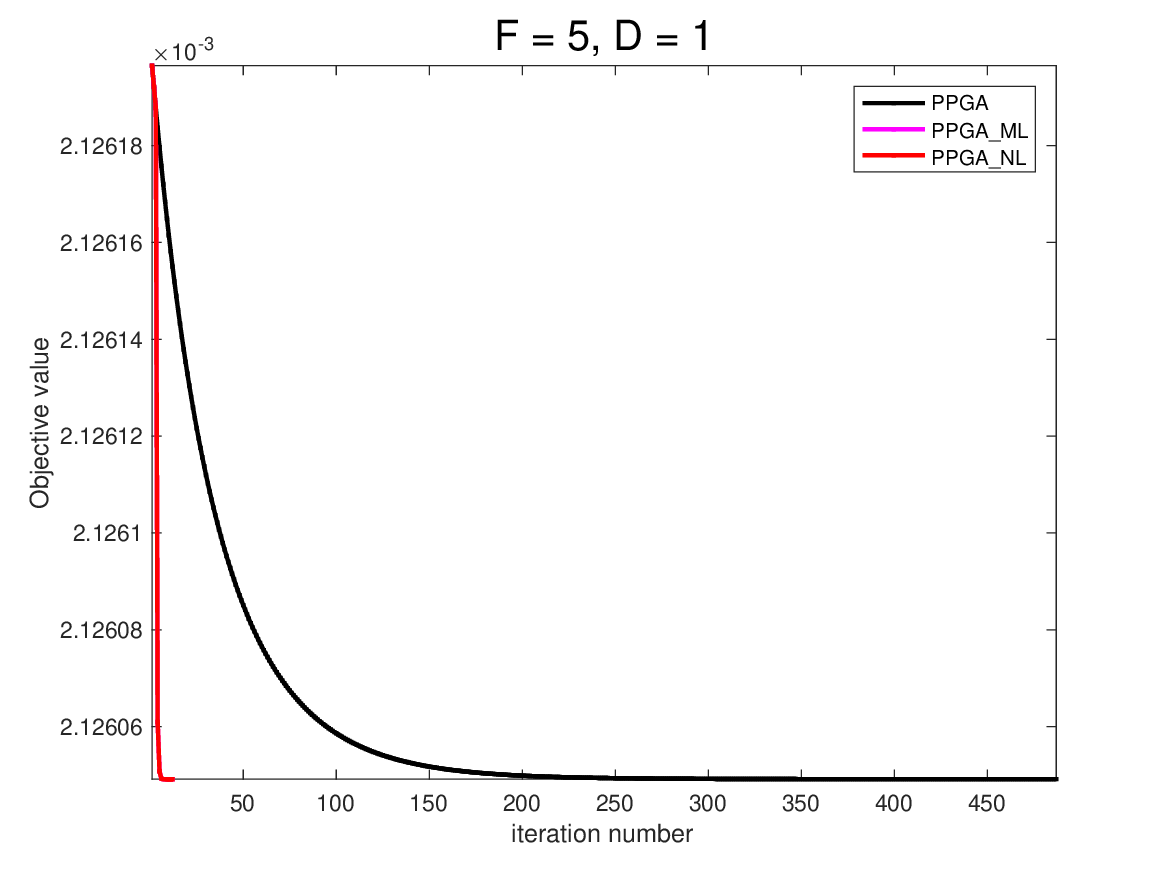}
\includegraphics[width=0.32\linewidth]{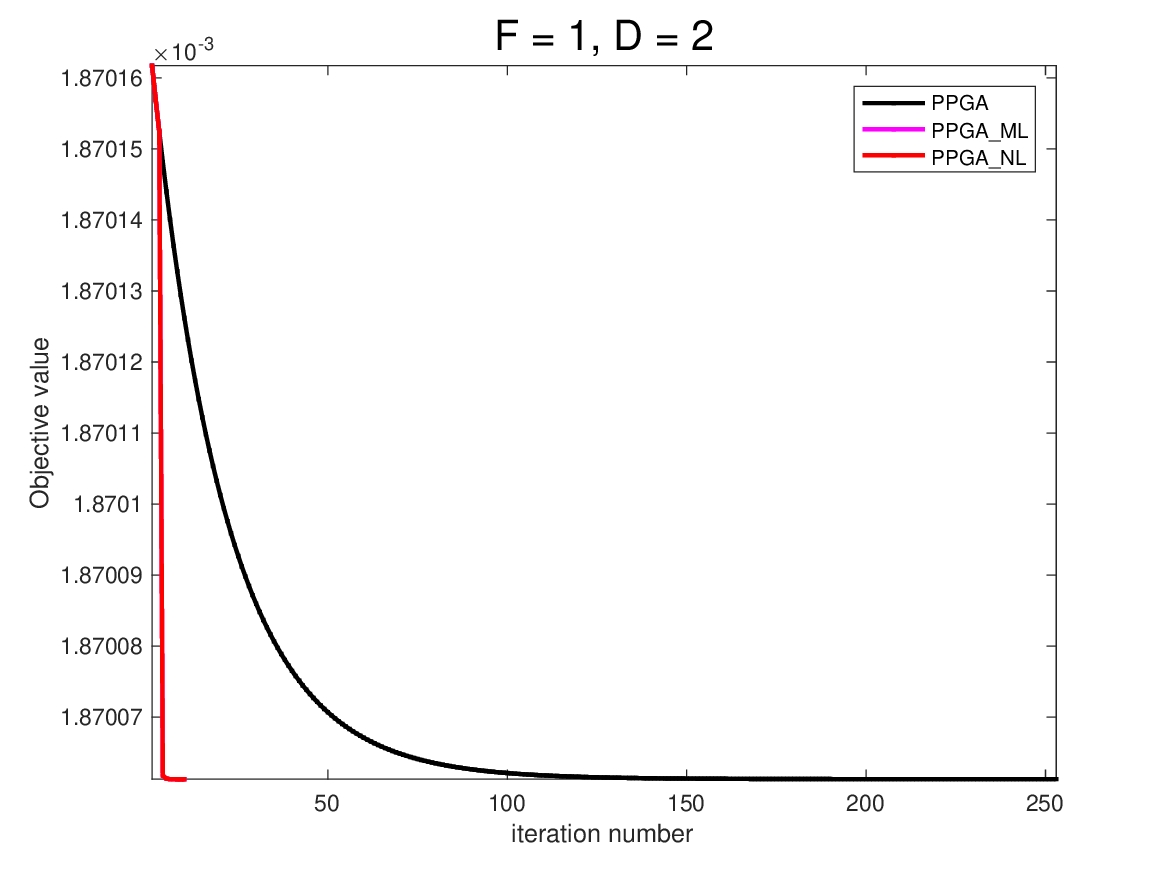} \hspace{5pt}
\includegraphics[width=0.32\linewidth]{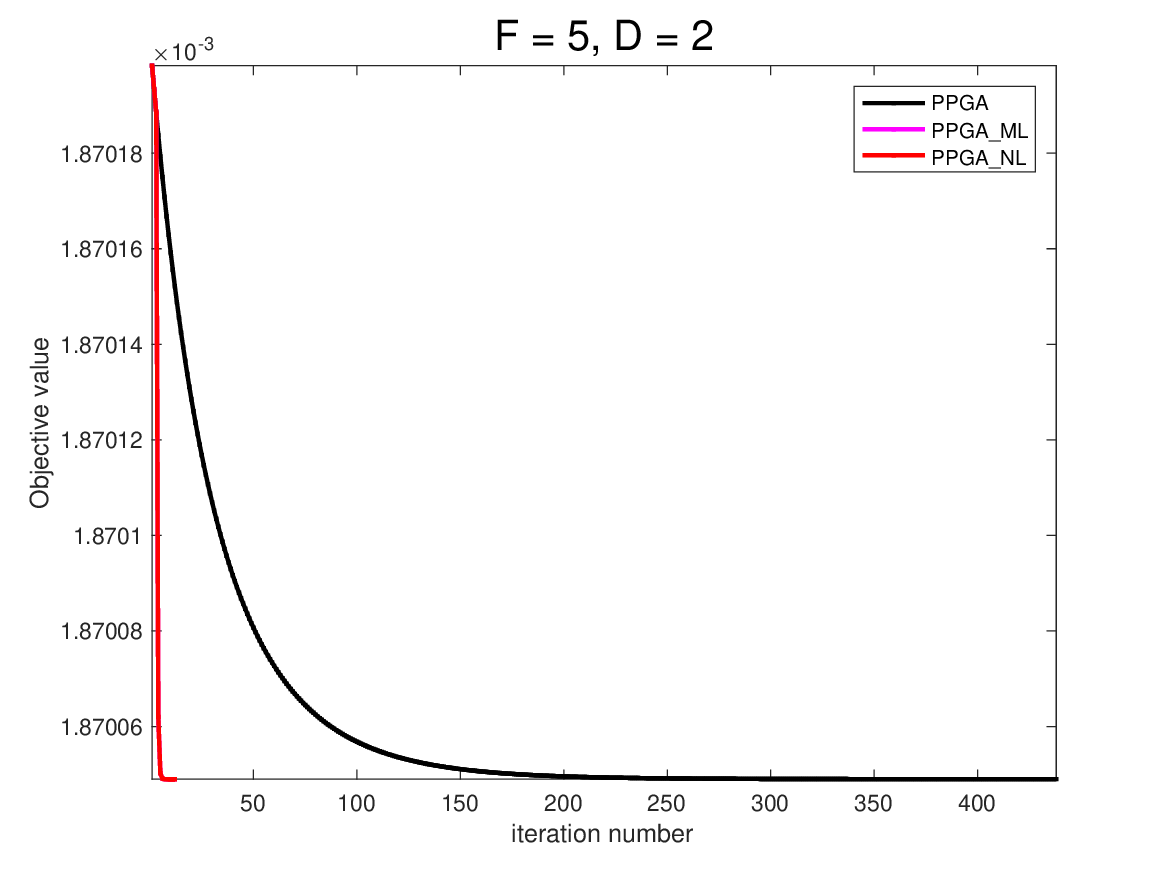}
\includegraphics[width=0.32\linewidth]{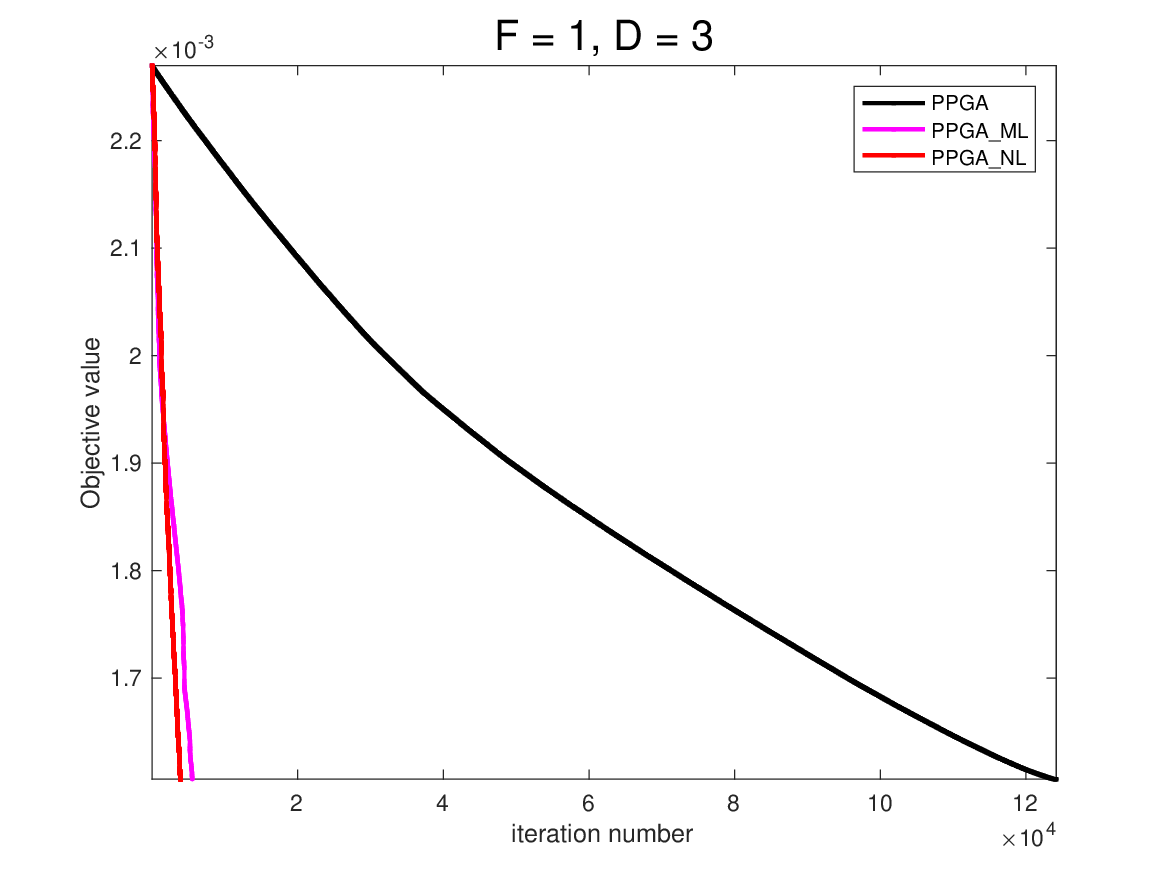} 
\includegraphics[width=0.32\linewidth]{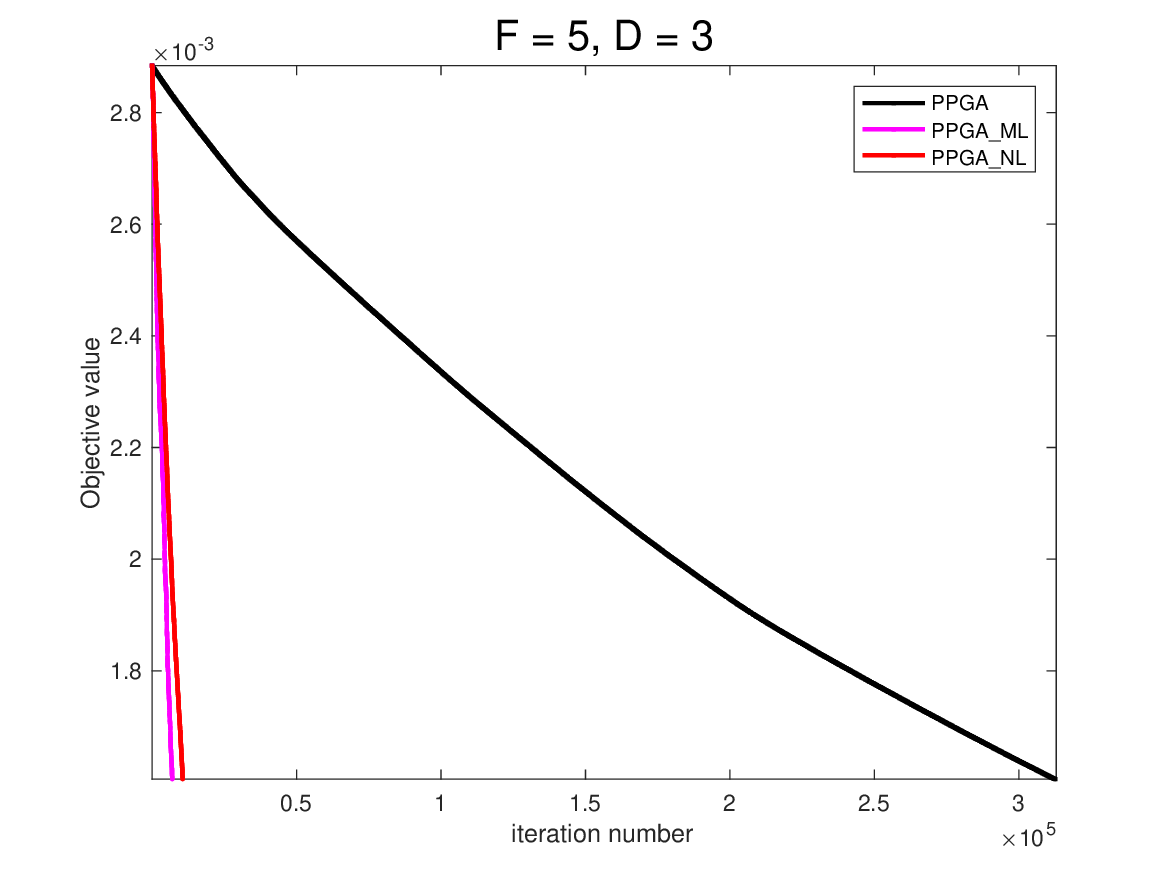}
\caption{Objective values of problem (\ref{l1_l2_minimization_problem_constraint_envelop}) versus iteration numbers.}\label{computation-efficient}
\end{figure}

\begin{figure}[!htbp]
\centering
\includegraphics[width=0.32\linewidth]{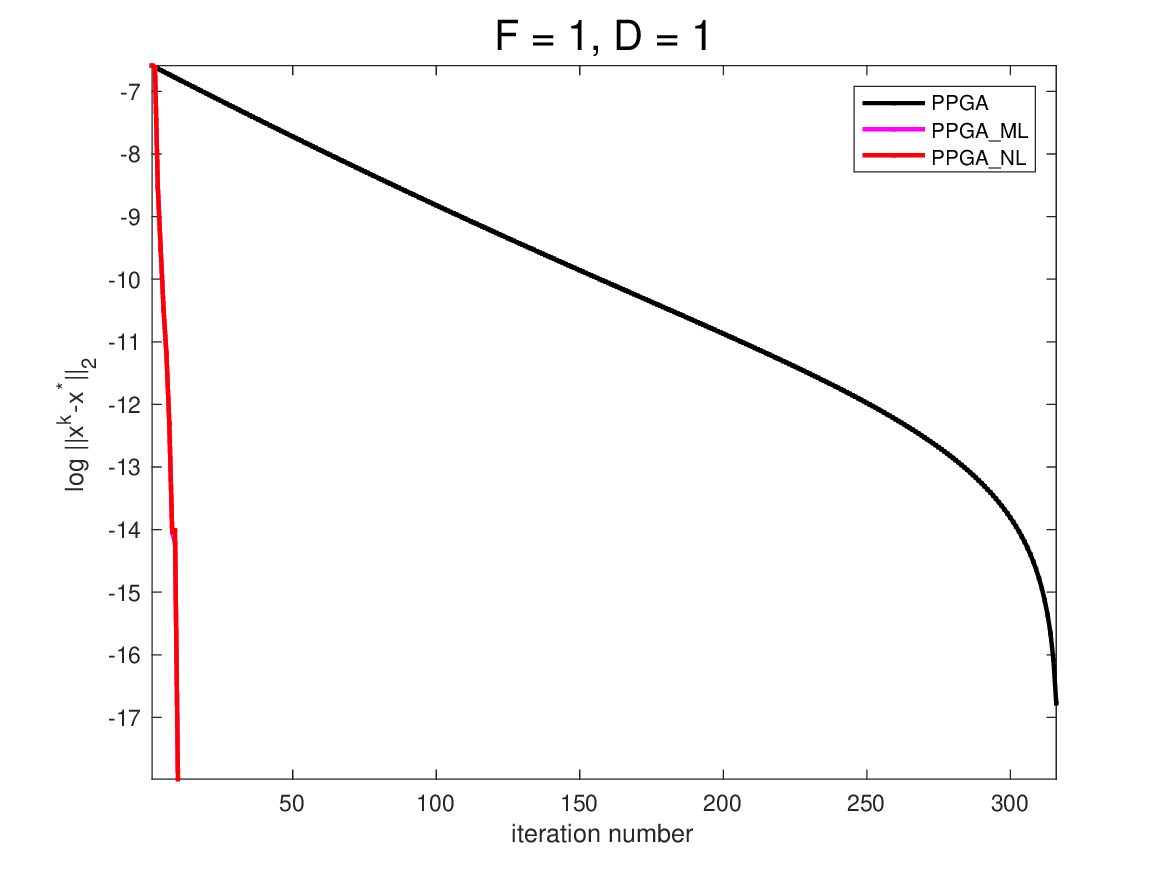}  
\includegraphics[width=0.32\linewidth]{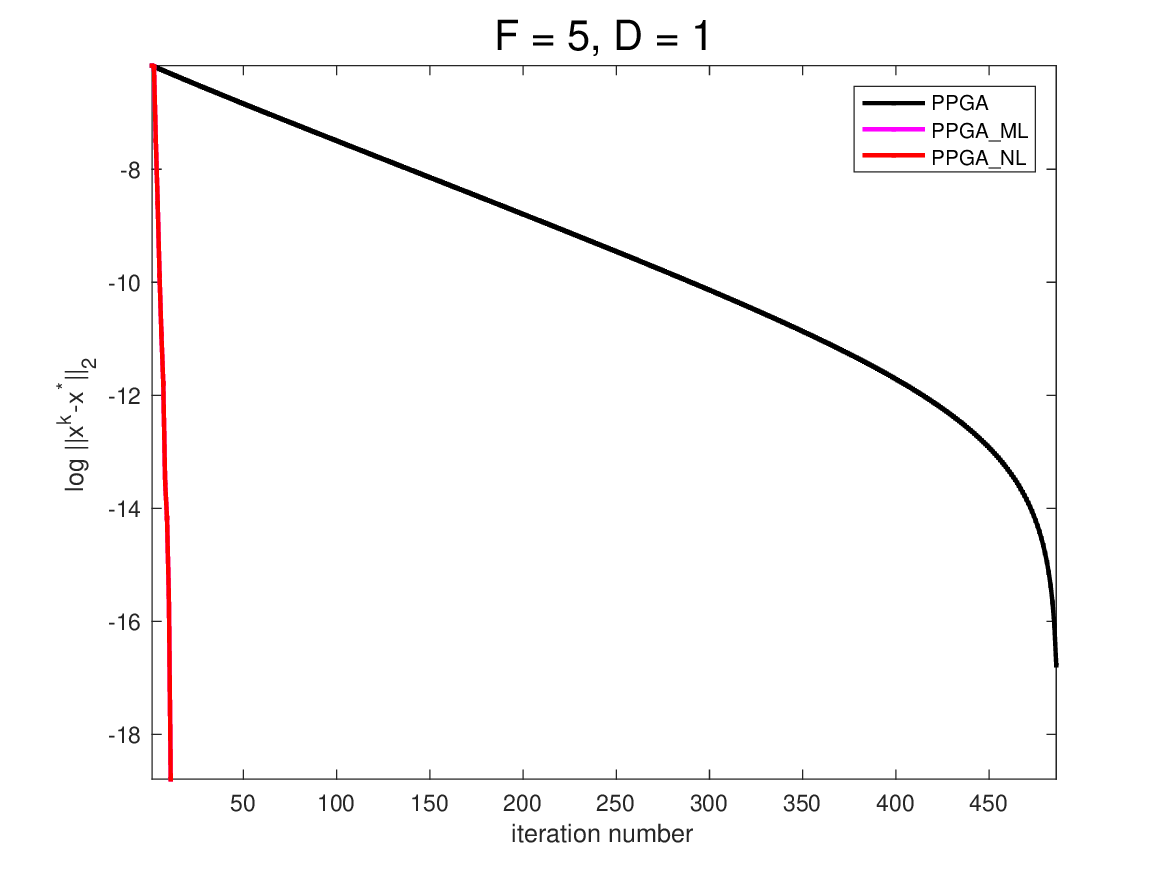}
\includegraphics[width=0.32\linewidth]{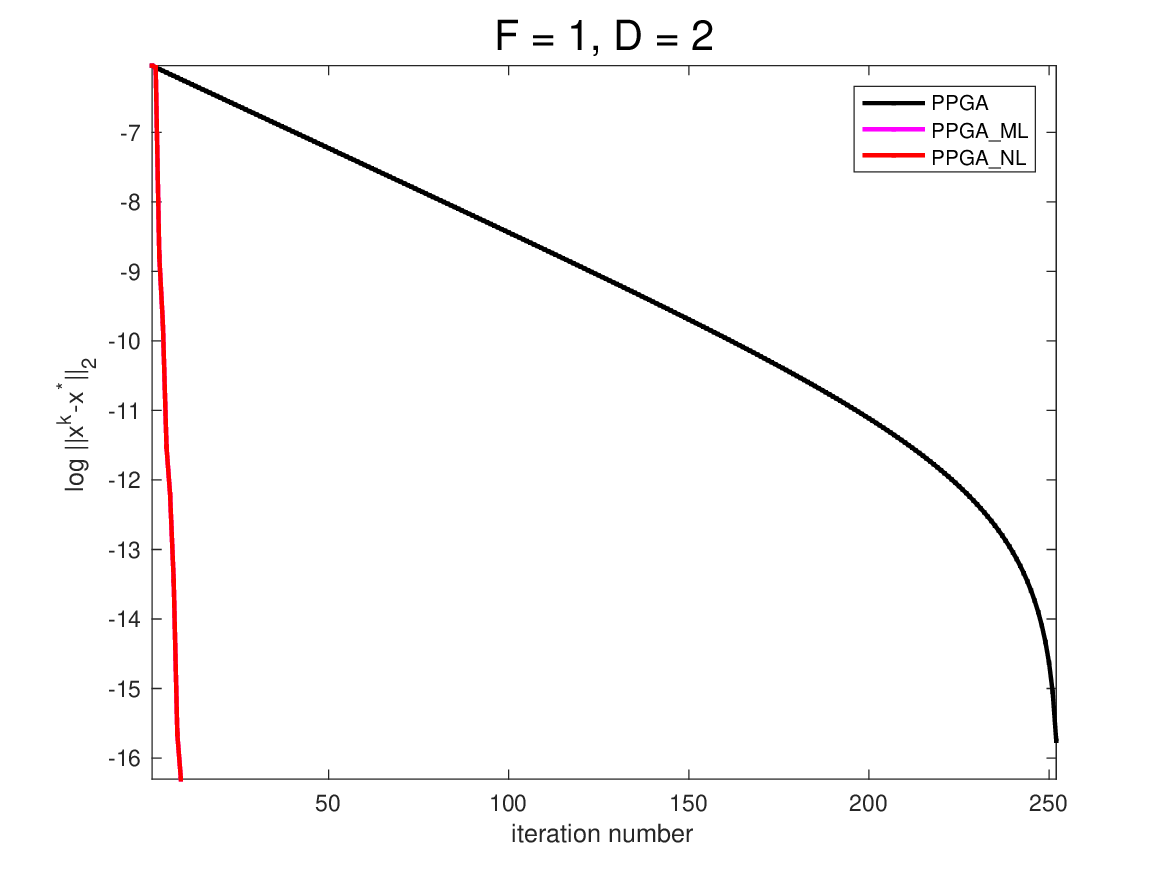} \hspace{5pt}
\includegraphics[width=0.32\linewidth]{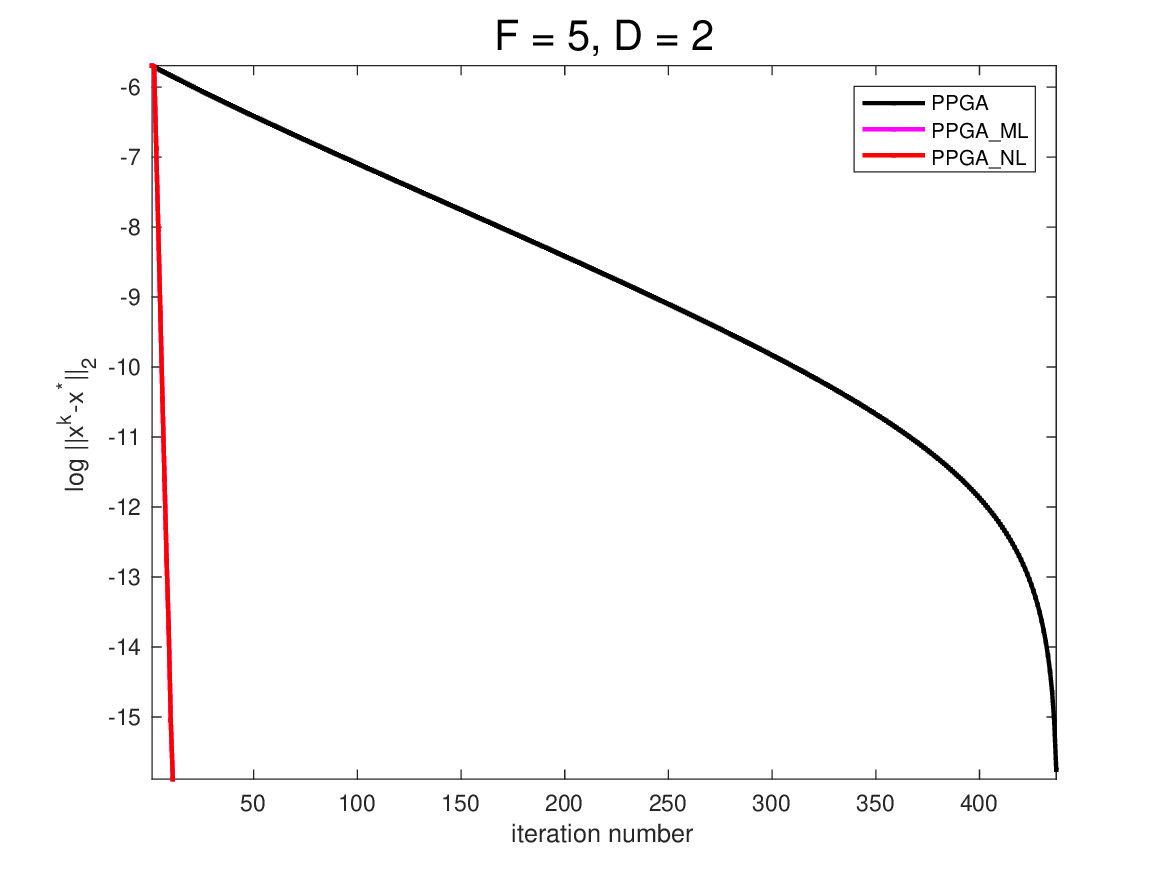}
\includegraphics[width=0.32\linewidth]{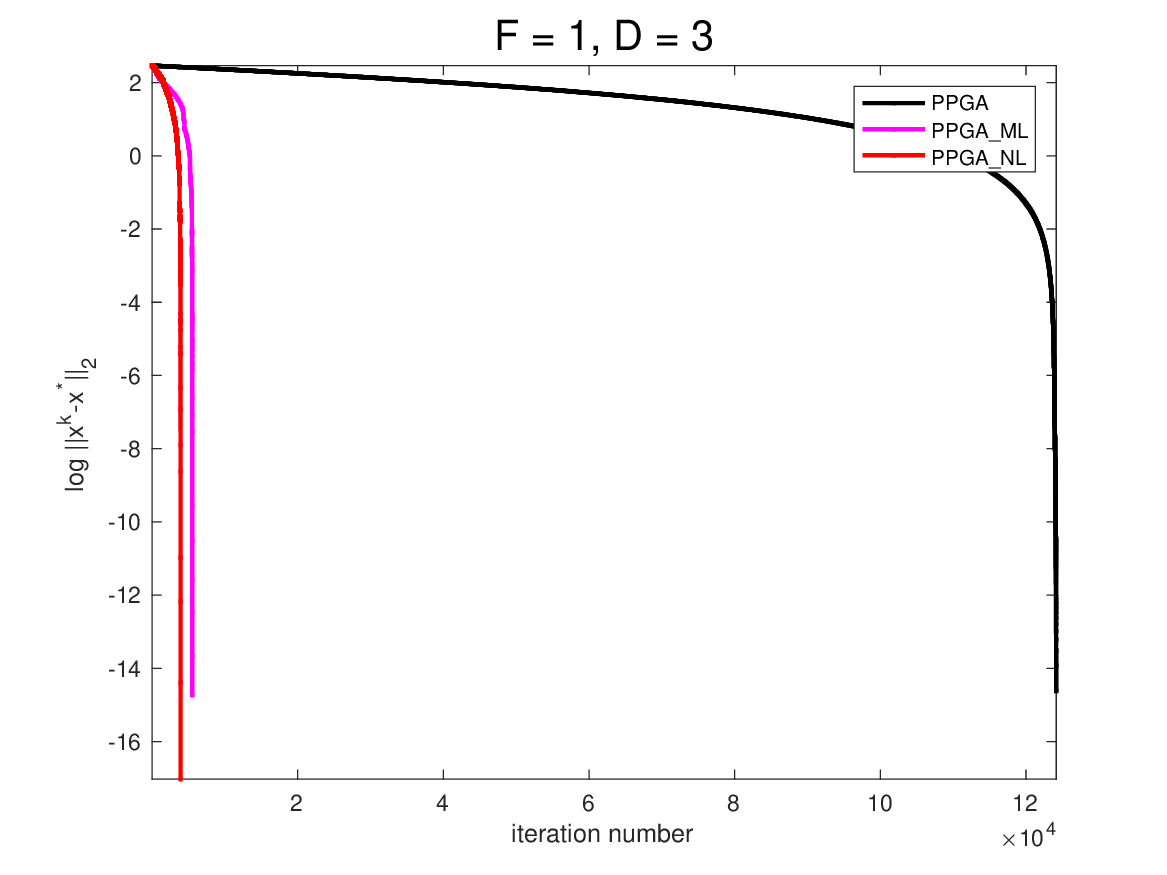}  
\includegraphics[width=0.32\linewidth]{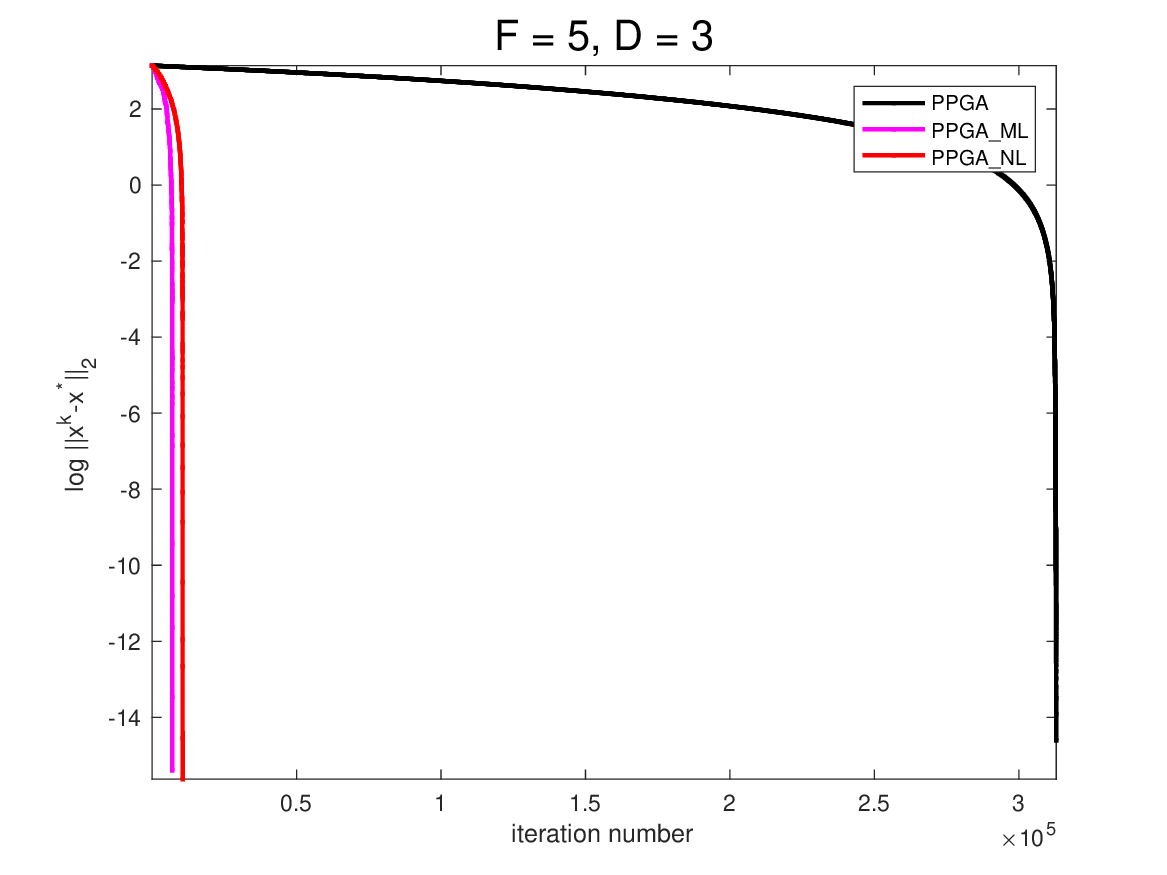}
\caption{$\|x^{k}-x^{\star}\|_{2}$ (in logarithmic scale).}\label{log_xk_x}
\end{figure}

We also plot $\|x^{k}-x^{\star}\|_{2}$ (in logarithmic scale) against the number of iteration in Figure \ref{log_xk_x}, where $x^{\star}$ is the approximated solution produced by the corresponding algorithm. It is obvious that the sequence generated by PPGA$\_$ML or PPGA$\_$NL converges much faster than that by PPGA. As can be seen from Figure \ref{log_xk_x}, the sequences generated by all the three algorithm appear to converge R-linearly although we have no theoritical results concerning the convergence rate of the algorithm.

\subsection{Noise-free case}
In this subsection, we focus on the $\ell_{1}/\ell_{2}$ based sparse signal recovery problem in the noise-free case for oversampled DCT matrices. Throughout this subsection, the oversampled DCT matrices have size of $64\times 1024$. Since the efficiency of the $\ell_{1}/\ell_{2}$ based model compared with other sparse promoting model has been extensively studied in \cite{Chao2019Accelerated}, we only compare the signal recovery capacity of our proposed algorithms with that of the state-of-the-art algorithms for solving the $\ell_{1}/\ell_{2}$ minimization problem (\ref{model_ratio}), including A1, A2, BS which are recently proposed in \cite{Chao2019Accelerated}, and $\ell_{1}/\ell_{2}$$\_$box \footnote{We use the Matlab code at https://github.com/yflouucla/L1dL2 with default parameter settings for A1, A2, BS and $\ell_{1}/\ell_{2}$$\_$box.} \cite{Rahimi-Wang-Dong-Lou} which applies ADMM to solve the $\ell_{1}/\ell_{2}$ minimization problem with a box constraint, i.e., problem (\ref{model_ratio}) with a box constraint 
$
\left\{x\in\mathbb{R}^{n}: d_{1}\le x_{i}\le d_{2}, i = 1,...,n\right\}.
$
We evaluate the performance of sparse signal recovery in term of success rate, that is the number of successful trials over the total number of trials. A success is declared if the relative error of the reconstructed solution $x^{\star}$ to the ground truth signal $x_{g}$ is no more than $10^{-3}$, i.e., $\frac{\|x^{\star}-x_{g}\|_{2}}{\|x_{g}\|_{2}}\le 10^{-3}$. In all the settings of this subsection, we set $\lambda = 0.001$ when $F \in \{1,5\}, D = 1$; $\lambda = 0.004$ when $F \in \{1,5\}, D = 3$; $\lambda = [0.004,0.004,0.1,0.2,0.5,0.5]$ as $s$ varies $[2,6,10,14,18,22]$ when $F \in \{1,5\}, D = 5$; the other parameters of PPGA, PPGA$\_$ML and PPGA$\_$NL are the same as those in subsection \ref{6_1}, and the parameter settings of A1, A2, BS and $\ell_{1}/\ell_{2}$$\_$box are chosen as suggested in \cite{Chao2019Accelerated} and \cite{Rahimi-Wang-Dong-Lou}.

Figure \ref{noisy-free} exhibits the success rate over 100 trials of all the compared algorithms for sparse recovery problem with $S\in\{2,6,...,22\}$, $F\in\{1,5\}$ and $D\in\{1,3,5\}$ for oversampled DCT matrices. Since PPGA, PPGA$\_$ML and PPGA$\_$NL obtain almost the same success rates in all the cases, we only include the results of PPGA$\_$NL for the sake of clarity of the plots. We observe from Figure \ref{noisy-free} that our proposed algorithm is comparable to A1, A2 and BS, and achieves much higher success rates compared with $\ell_{1}/\ell_{2}\_$box.

\begin{figure}[!htbp]
\centering
\includegraphics[width=0.45\linewidth]{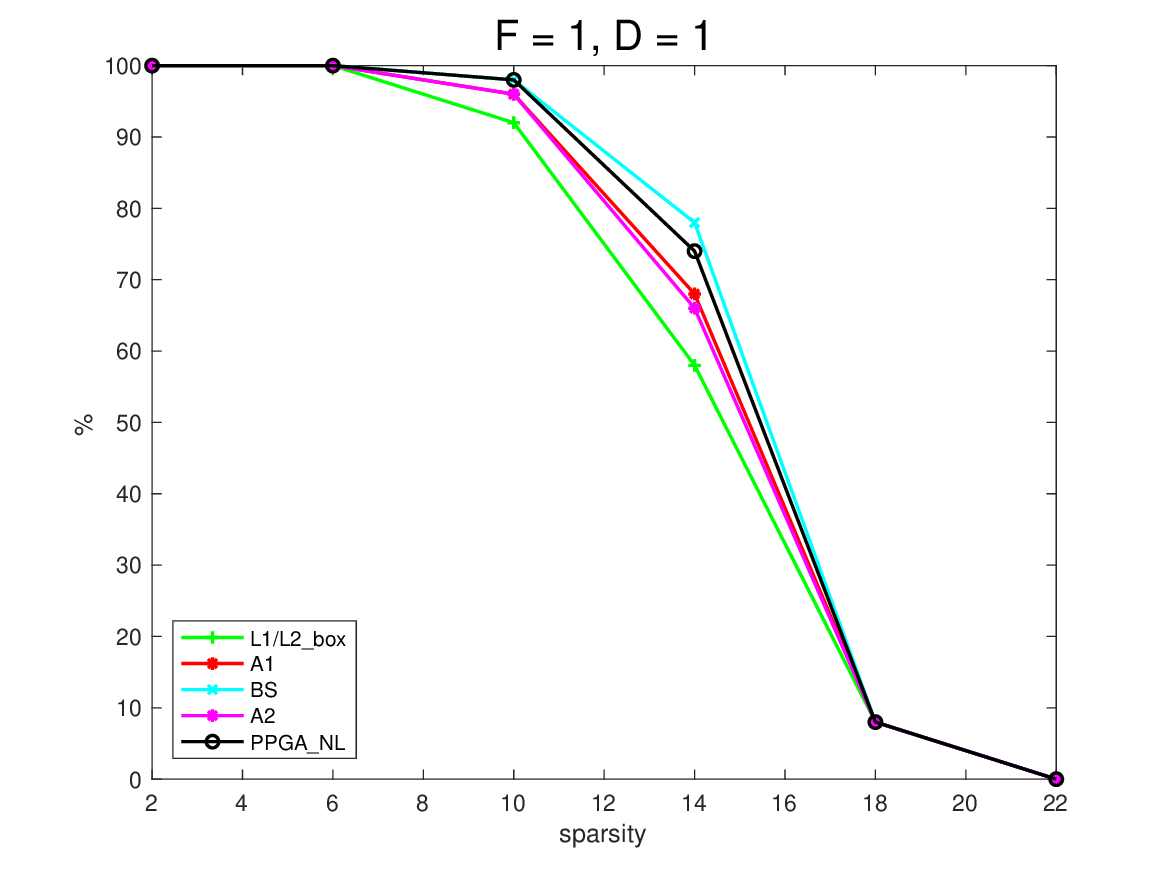} \hspace{5pt}
\includegraphics[width=0.45\linewidth]{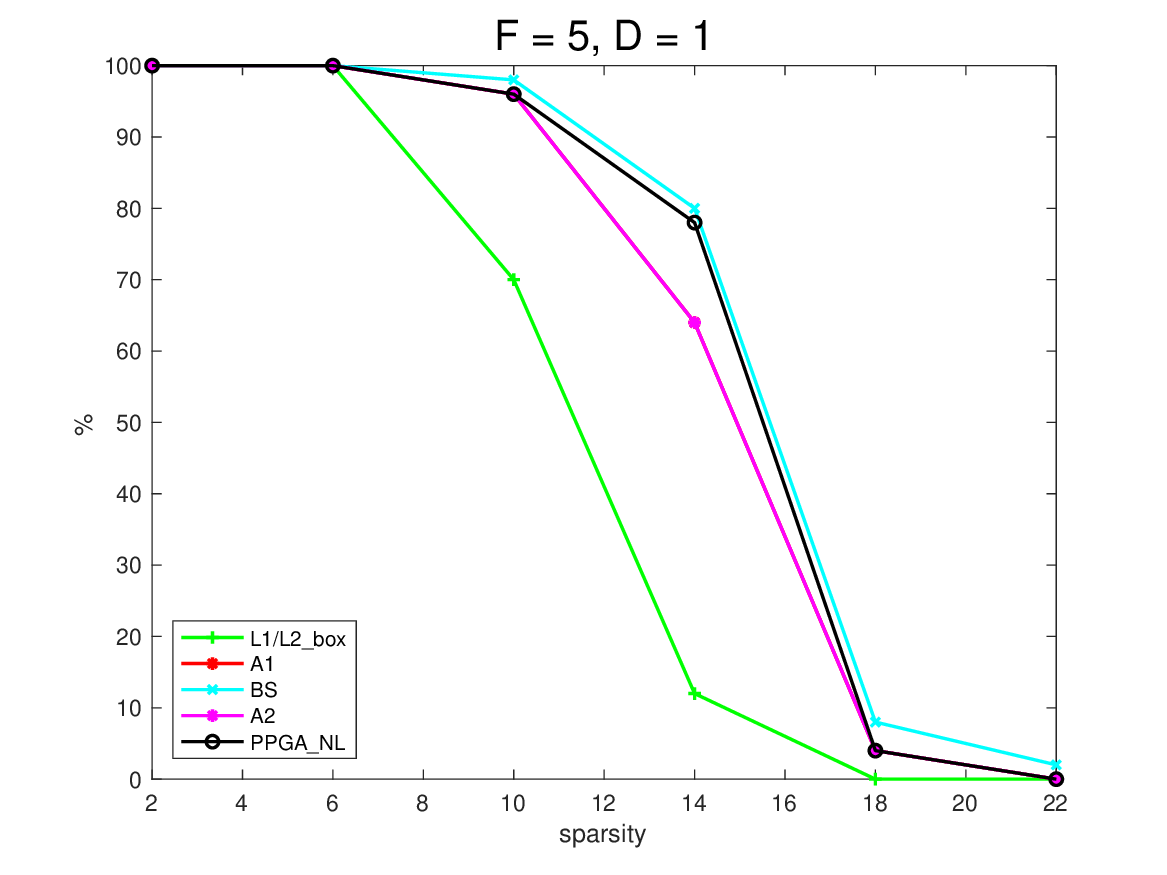}
\includegraphics[width=0.45\linewidth]{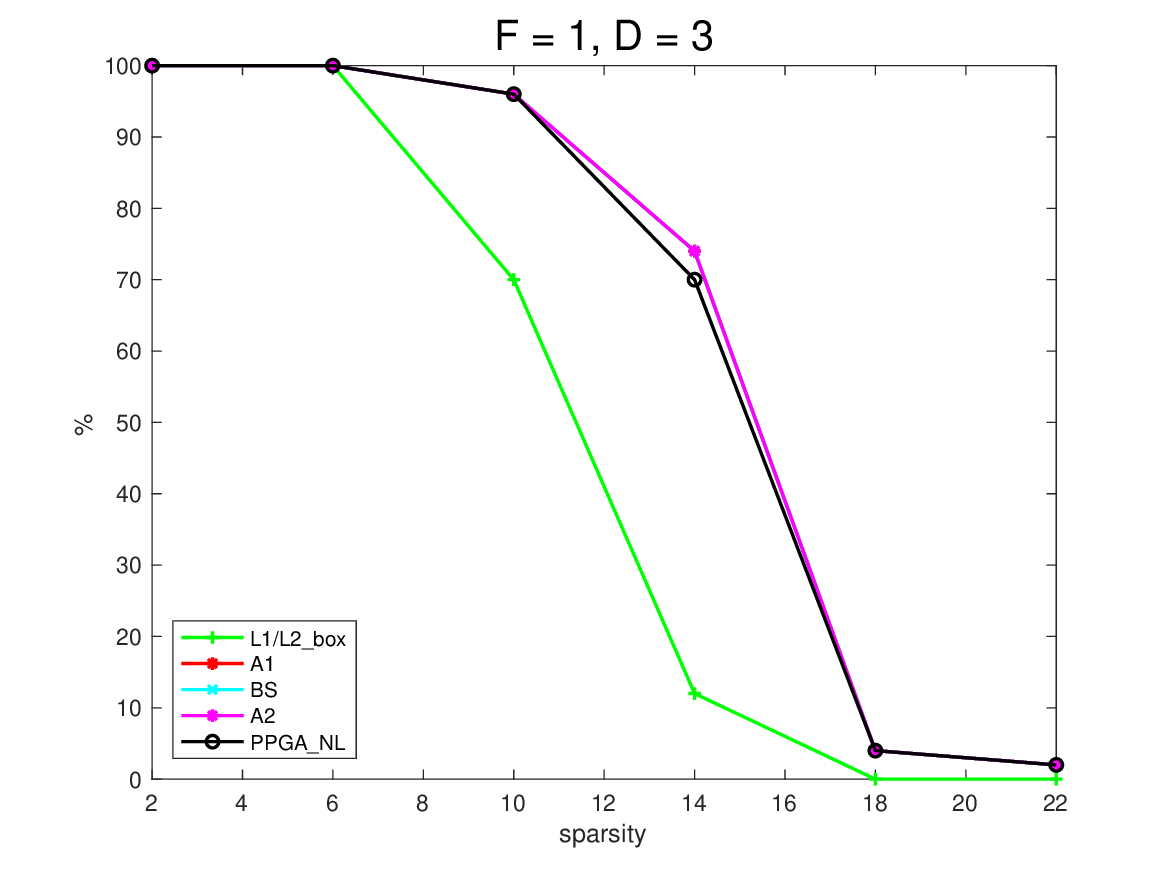} \hspace{5pt}
\includegraphics[width=0.45\linewidth]{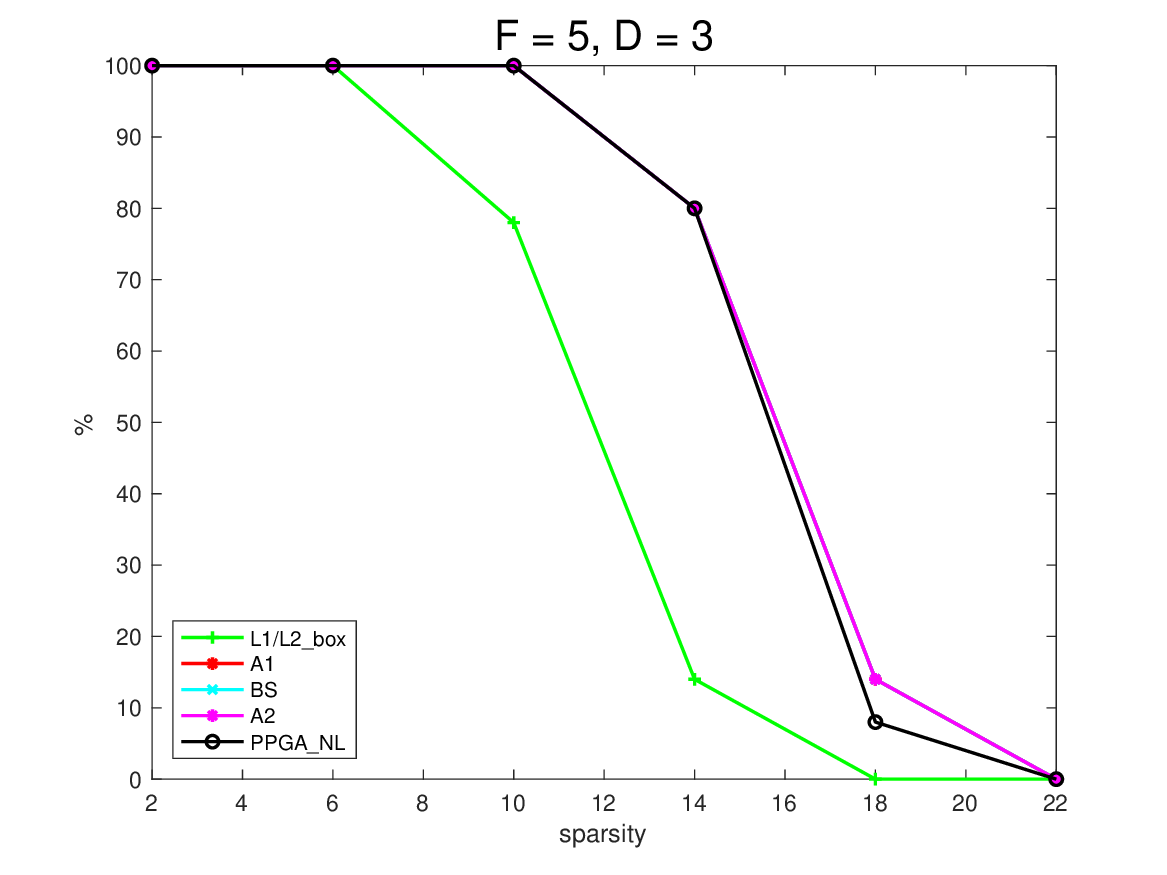}
\includegraphics[width=0.45\linewidth]{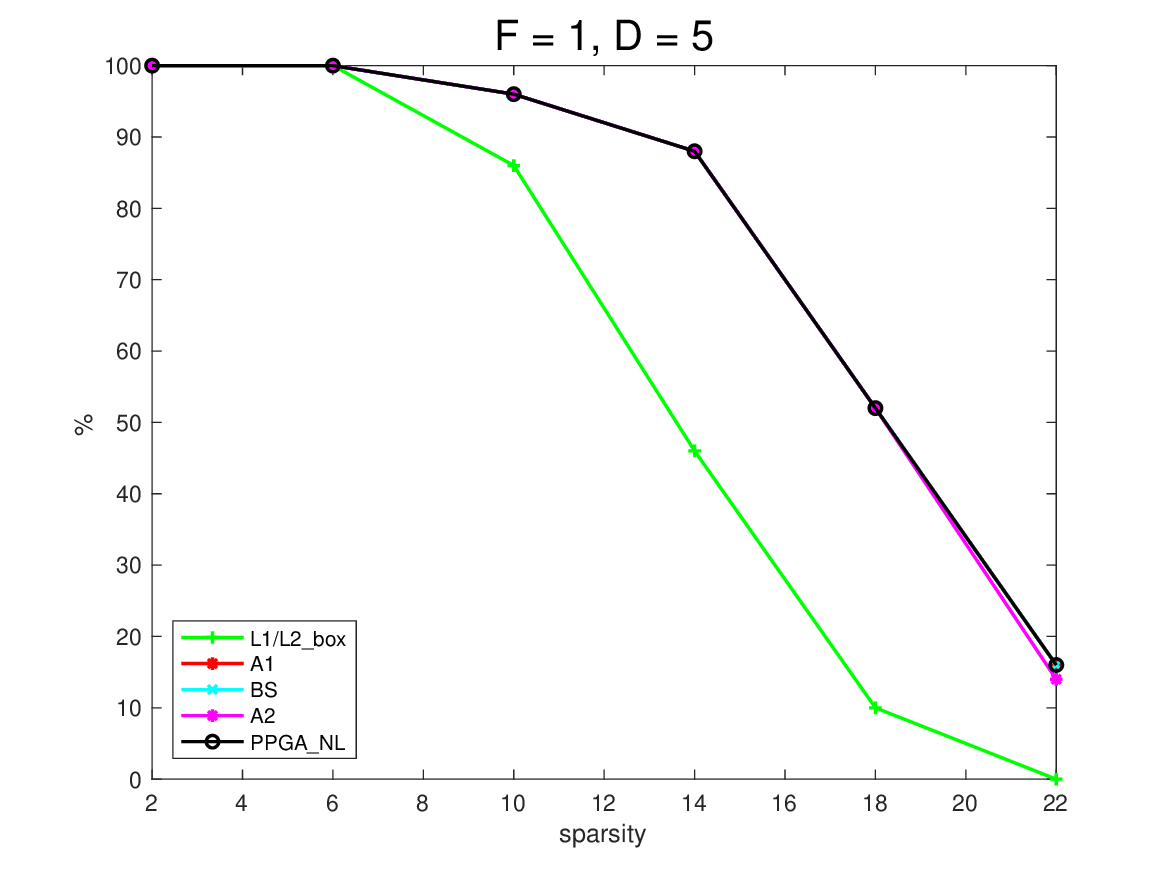} \hspace{5pt}
\includegraphics[width=0.45\linewidth]{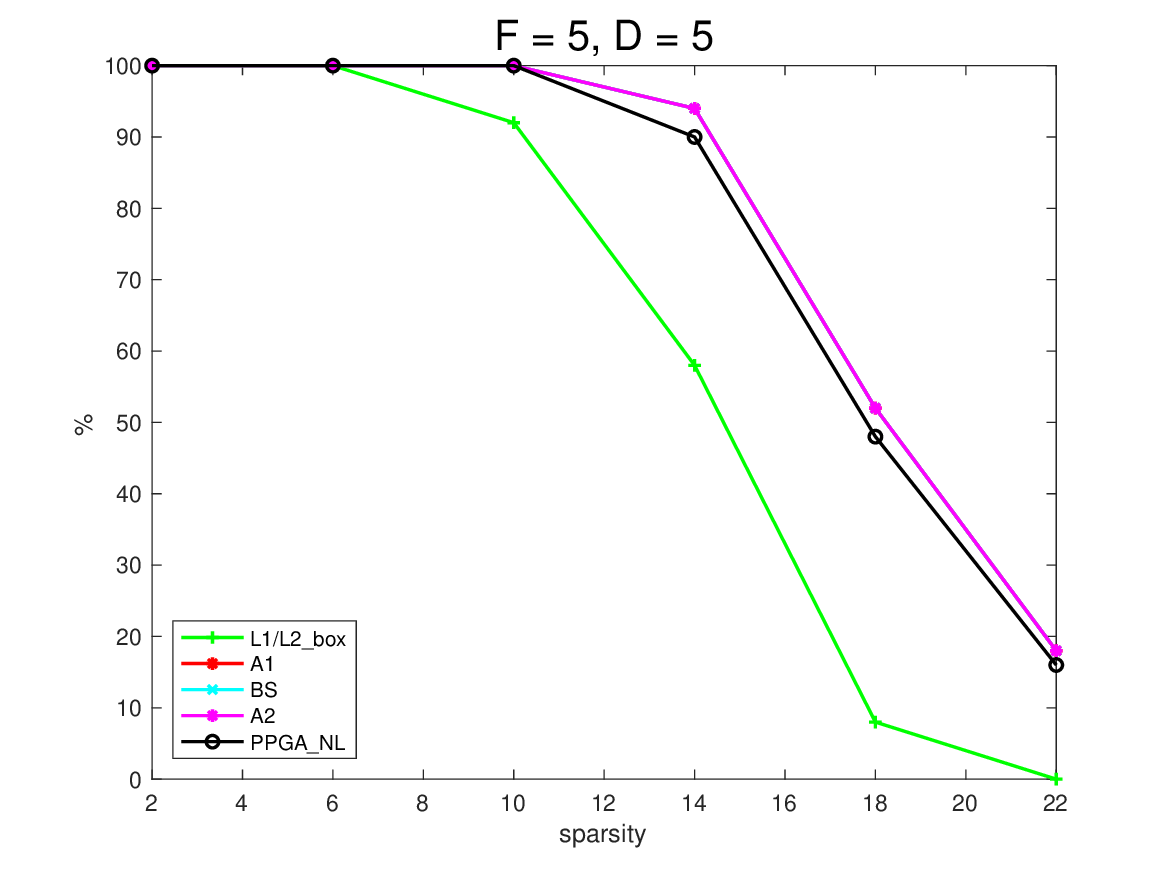}
\caption{Plots of success rates versus sparsity for noise-free sparse recovery with oversampled DCT matrices.}\label{noisy-free}
\end{figure}

\subsection{Noisy data}
In this subsection, we consider the sparse signal recovery problem with white Gaussian noise. In order to test the performance of the proposed algorithms in solving $\ell_1/\ell_2 $ based sparse recovery problems, we present in subsection \ref{Algorithmic comparison} various computational aspects of the proposed algorithms together with comparison to MBA \cite{Zeng2021Analysis}. Then,  we compare various sparse promoting models in subsection \ref{6_3_2}.
\subsubsection{Algorithmic comparison}\label{Algorithmic comparison}
We compare our proposed algorithms with MBA \footnote{We use the Matlab code at https://www.polyu.edu.hk/ama/profile/pong/MBA$\_$l1vl2/ with default parameter settings for MBA.} \cite{Zeng2021Analysis} for $\ell_1/\ell_2$ sparse signal recovery with white Gaussian noise in this subsection. The test instances and experiment settings are similar as those in \cite{Zeng2021Analysis}. Specifically, A is an $m\times n$ random oversampled DCT matrix, the ground truth $x_{g}$ is generated as (\ref{groundtruth}), and $b = Ax_{g}+0.01e$, where $e\in\mathbb{R}^{m}$ has i.i.d standard Gaussian entries. Here, $m = 64, n = 1024, s\in\{8,12\}, F\in \{5,15\}$ and $D\in\{2,3\}$. The initial points of the compared algorithms in this subsection are chosen as 
$$
x^{0} = \left\{
\begin{aligned}
&A^{\dagger}b+\epsilon\frac{x_{\ell_{1}}-A^{\dagger}b}{\|Ax_{\ell_{1}} - b\|_{2}}, \qquad {\rm if} \ \|Ax_{\ell_{1}} - b\|_{2}>\epsilon,\\
&x_{\ell_{1}}, \qquad  \qquad  \qquad \qquad \quad {\rm otherwise},
\end{aligned}
\right.
$$
where $x_{\ell_{1}}$ is an approximate solution of (\ref{initial_solution}) and $A^{\dagger}$ is the pseudoinverse of $A$. We point out that such $x^{0}$ does in the constraint of problem (\ref{model_ratio_constraint}). The recovery error 
$
{\rm ReeErr} = \frac{\|x^{\star} - x_{g}\|_{2}}{{\rm max}\{1,\|x_{g}\|_{2}\}},
$
with $x^{\star}$ being the output of the corresponding algorithm, is used to evaluate the recovery capacity of the algorithm. For
our proposed algorithms, $\epsilon$ is set to be $3\times 10^{-3}\sqrt{m}$ and all the other parameters are chosen as those in subsection \ref{6_1}. For MBA, we consider two parameter settings. The first one, referred to as MBA$\_$1, chooses the same $\epsilon$ as our proposed algorithms, i.e., $\epsilon=3\times 10^{-3}\sqrt{m}$. The second one, referred to as MBA$\_$2, takes $\epsilon=1.2\|0.01e\|_2$, which is as the same as that in \cite{Zeng2021Analysis}.  All the other parameters for MBA$\_$1 and MBA$\_$2 are determined as suggested in \cite{Zeng2021Analysis}. 

\renewcommand{\arraystretch}{1.2}
\begin{table}[htp]
	\caption{\label{Algorithm comparation} CPU, ReeErr for Gaussian noise recovery with oversampled DCT matrices.} 
	\resizebox{\textwidth}{36mm}{
		\begin{tabular}{cccccccc} \hline
			s\quad F\quad D& &spg$\ell_{1}$&PPGA$\_$ML&PPGA$\_$NL&MBA$\_$1&MBA$\_$2\cr  \hline
			8\quad 5\quad 2 
			&CPU&0.1551 &0.0968 &0.0841 &1.0518 &0.2143 \cr    
			&ReeErr&2.8587e-02&\textbf{1.9687e-03}&\textbf{1.9687-03}&2.3131e-03&2.2823e-03\cr    
			8\quad 5\quad 3	
			&CPU& 0.1231&0.1315&0.1104&0.8697&0.1933\cr   
			&ReeErr&4.4835e-03&6.4040e-04&\textbf{6.4038e-04}&6.4730e-04&6.7796e-04\cr  
			8\quad 15\quad 2
			&CPU&0.1575&1.7223&2.2918&200.7038&4.2432\cr   
			&ReeErr&4.6736e-01	&1.4949e-01&1.4949e-01&1.4906e-01&\textbf{1.4813e-01}\cr
			8\quad 15\quad 3
			&CPU& 0.1894 &9.2470 &13.5070 &563.3438 &32.4322 \cr   
			&ReeErr& 3.3159e-01&3.9199e-02&3.9198e-02&4.0596e-02&\textbf{3.7558e-02}\cr
			12\quad 5\quad2
			&CPU& 0.1364 &0.7952 &0.7834 &9.7543 &2.8424 \cr   
			&ReeErr& 1.3524e-01&\textbf{3.6374e-02}&\textbf{3.6374e-02}&3.6647e-02	&3.6384e-02\cr       
			12\quad5\quad3	
			&CPU& 0.1614 &2.9272 &3.4987 &145.3295 &6.7533 \cr   
			&ReeErr& 5.9235e-02&3.7540e-03&3.7566e-03&\textbf{3.2676e-03}&3.8409e-03\cr  
			12\quad15\quad2	
			&CPU& 0.1729 &2.6958 &3.1855 &120.9879 &9.4324 \cr    
			&ReeErr& 5.7682e-01&\textbf{1.6896e-01}&\textbf{1.6896e-01}&1.6960e-01&2.0444e-01\cr 
			12\quad15\quad3   
			&CPU& 0.1906 &221.3362 &31.9408 &1459.4108 &359.7682 \cr  
			&ReeErr& 5.2954e-01&\textbf{5.9615e-01}&6.3612e-01&6.9707e-01&1.4882e+00\cr  \hline
	\end{tabular}}
\end{table} 

Table \ref{Algorithm comparation} presents the computational results (averaged over the 20 trials) including the CPU time and the recovery error of all compared algorithms. Since PPGA costs too much computational time when $(s,F,D) = (12,15,3)$, we do not include PPGA here. It can be observed from Table \ref{Algorithm comparation} that compared with MBA, our proposed PPGA$\_$ML and PPGA$\_$NL spend less CPU time to converge and obtain recovered signals with smaller recovery error in most cases.
\subsubsection{Comparision with other sparsity promoting models}\label{6_3_2}
This subsection is devoted to the comparison on different sparse promoting models, including $\ell_{1}/\ell_{2}$, $\ell_{1}$, $\ell_{1}-\ell_{2}$ and $\ell_{1/2}$ for Gaussian noise sparse recovery with Gaussian matrices. Here, the $\ell_{1}$, $\ell_{1}-\ell_{2}$ and $\ell_{1/2}$ models are the least square problems with the corresponding sparse promoting regularizers. The solutions of the $\ell_{1/2}$ and $\ell_{1}-\ell_{2}$ are obtained by the half-thresholding method \cite{Xu2012regularization} and the forward-backward splitting (FBS) method \footnote{We use the Matlab code at https://github.com/mingyan08/ProxL1-L2 with default parameter settings for $\ell_{1}-\ell_{2}$ (FBS).} \cite{Yifei2018Fast}, respectively.

The experimental setup follows that of \cite{Yifei2018Fast}. Specifically, we consider a signal $x$ of length $n = 512$ and $s = 130$ non-zero elements. The number of measurements $m$ varies from 230 to 330. The sensing matrix $A$ is generated by normalizing each column of a Gaussian matrix to be zero-mean and unit norm. The stand deviation of the white Gaussian noise is set to be $\sigma = 0.1$. We use the mean-square-error (MSE) to quantify the recovery performance, i.e., MSE = $\|x^{\star} - x_{g}\|_{2}$, where $x^{\star}$ is the output of the corresponding algorithm. If the support of the ground truth signal is known, denoted as $\Lambda = {\rm supp}(x_{g})$, one can compute the MSE of an oracle solution as $\sigma^{2}{\rm tr}(A_{\Lambda}^{T}A_{\Lambda})^{-1}$ as benchmark, where $A_{\Lambda}$ denotes the columns of matrix $A$ restricted to $\Lambda$.

In our experiments, the parameters for $\ell_1-\ell_2$ and $\ell_1$ models are chosen as the same as those in \cite{Yifei2018Fast}. The parameters of $\ell_{1/2}$ model are determined as suggested in \cite{Xu2012regularization} and the parameter $\mu$ is set to be 0.05. For the proposed PPGA$\_$ML and PPGA$\_$NL, we set $\epsilon=0.05\sqrt{m}$, $d = 10^{7}$, $\eta=0.5$, $a=10^{-8}$ and $N=4$. Since  the proposed algorithms are designed for solving problem (\ref{l1_l2_minimization_problem_constraint_envelop}) which is an approximate model to model (\ref{l1/l2_minimization_problem_constraint}), the smaller $\lambda$ in model  (\ref{l1_l2_minimization_problem_constraint_envelop}) is chosen, the closer model (\ref{l1_l2_minimization_problem_constraint_envelop}) to model (\ref{l1/l2_minimization_problem_constraint}). Hence, we initial set $\lambda=0.01$, then half it in every ten iterations and fix it when the total number of iterations exceeds 500. 

Figure \ref{noisy2} presents the average MSE of the compared algorithms over 20 instances. One can observe that the $\ell_{1}/\ell_{2}$ based model (solved by PPGA$\_$ML and PPGA$\_$NL) achieves the smallest MSE among all the competing sparse promoting models when $m\ge 250$.  Both PPGA$\_$ML and PPGA$\_$NL perform very similarly in this experiment. 
As considered in \cite{Xu2012regularization}, we show in Table \ref{computation time for the Gaussian noise} the mean and standard deviation of MSE as well as computational time of the compared algorithms at $m = 238, 250, 276, 300$. It is demonstrated in Table \ref{computation time for the Gaussian noise} that the proposed algorithms spend much less CPU time to converge while achieve smaller MSE than other compared models when $m=250, 276, 300$.

\begin{figure}[!htbp]
\centering
\includegraphics[width=0.8\linewidth]{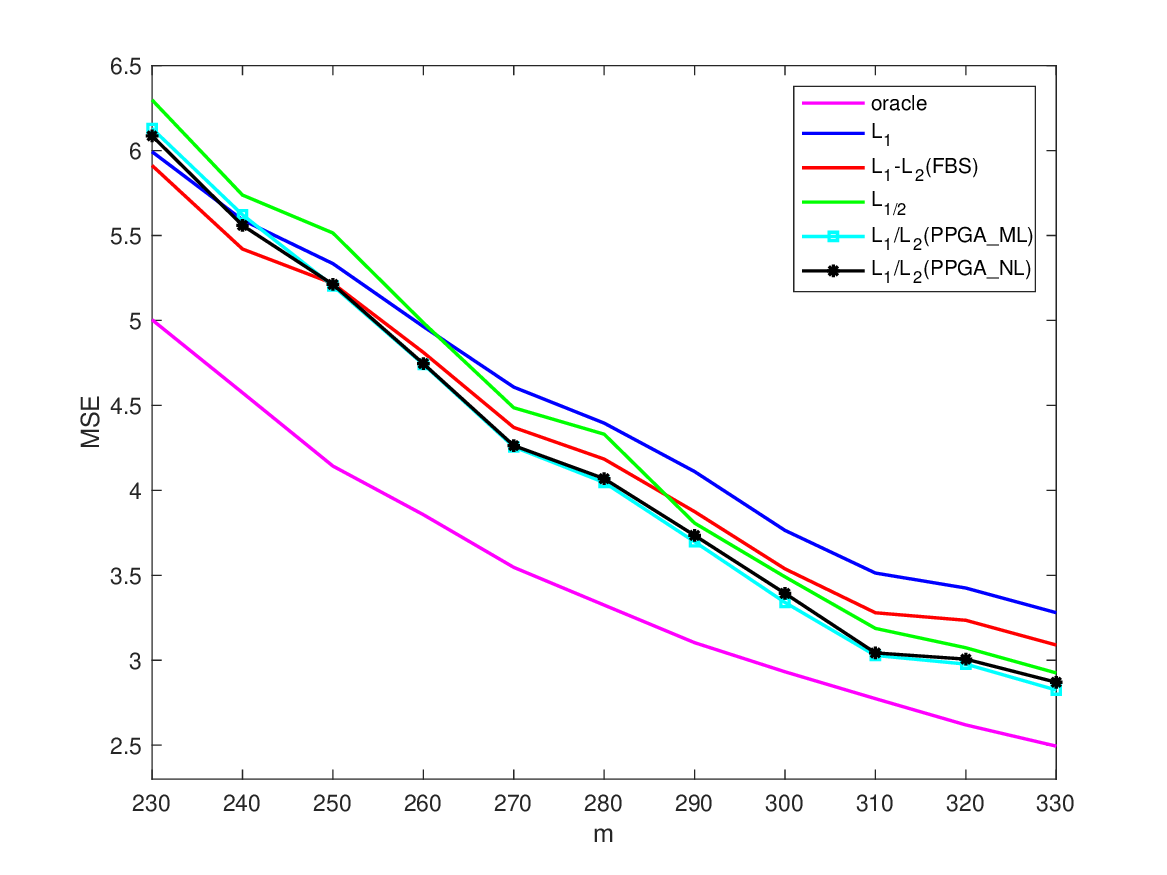} 
\caption{Plots of MSE versus $m$ for the Gaussian noise case with Gaussian matrices.}\label{noisy2}
\end{figure}

\renewcommand{\arraystretch}{1.3}
\begin{table}[htp]
\caption{\label{computation time for the Gaussian noise} MSE and CPU for Gaussian noise recovery with Gaussian matrices.} 
\resizebox{\textwidth}{28mm}{
\begin{tabular}{ccccccc} \hline
    Methods&m&MSE&CPU&m&MSE&CPU  \cr\hline
    Oracle&                    &  4.6338(0.0532)                        &                    &   & 4.1429(0.0669)                    &\cr   
    $\ell_{1}$&&5.8823(0.5491)&0.5773(0.1290)&&0.5593(0.0979)&5.4922(0.6880)   \cr               
    $\ell_{1}/\ell_{2}$(PPGA$\_$ML)&       & 5.6732(0.6699) & 0.1031(0.0164) &  & \textbf{5.2034}(0.8307) & 0.1298(0.0240)\cr
    $\ell_{1}/\ell_{2}$(PPGA$\_$NL)&   238   & 5.6612(0.6423) & \textbf{0.0634}(0.0171) & 250 & 5.2121(0.8217) & \textbf{0.0699}(0.0250)\cr
    $\ell_{1}-\ell_{2}$(FBS)& & \textbf{5.5579}(0.5955) & 0.9048(0.4324)& & 5.2177(0.7367) & 0.9007(0.4227)\cr
    $\ell_{1/2}$&                 & 5.8895(0.6915) & 5.3404(1.9407)&  & 5.5149(0.8703) & 4.9879(1.9255)\cr
    Oracle&                    &    3.4251(0.0430)  &                    &   & 2.9320(0.0246)  &\cr
     $\ell_{1}$&&4.6324(0.6524)&0.6700(0.0536)&&3.8959(0.5256)&1.0740(0.6979)   \cr               
    $\ell_{1}/\ell_{2}$(PPGA$\_$ML)&     & \textbf{4.1928}(0.7112) & 0.1491(0.0339) &  & \textbf{3.3343}(0.4409) & 0.1379(0.0136)\cr
    $\ell_{1}/\ell_{2}$(PPGA$\_$NL)&    276     & 4.2022(0.6985) & \textbf{0.0948}(0.0402)&  300 & 3.3618(0.4621) & \textbf{0.0778}(0.0210)\cr
    $\ell_{1}-\ell_{2}$(FBS)& & 4.2906(0.6378) & 0.7995(0.3518)& & 3.5378(0.5497) & 0.3443(0.1047)\cr
    $\ell_{1/2}$&                 & 4.3665(0.6693) & 4.4069(1.4127)&  & 3.4908(0.5393) & 3.5280(1.2859)\cr   
        \hline
        \end{tabular}}
\end{table}

\section{Conclusions}\label{S7}
In this paper, we study algorithms for solving $\ell_{1}/\ell_{2}$ sparse signal recovery problems. A commonly used $\ell_{1}/\ell_{2}$ model for Gaussian noise compressed sensing is to minimize $\ell_{1}/\ell_{2}$ over an elliptic constraint. We first propose to smooth the indicator function on the elliptic constraint by a smooth fractional function. It is proven that stationary points of the proposed model tend to those of the elliptically constrained problem as the smoothing parameter tends to zero. Inspired by the parametric method of the single-ratio fractional programming, we develop a parameterized proximal-gradient algorithm (PPGA) and its line search counterpart (PPGA$\_$L) for solving the proposed problem. The global convergence of the sequences generated by PPGA and PPGA$\_$L with monotone objective values is established. Numerical experiments demonstrate that our proposed algorithms is efficient for $\ell_{1}/\ell_{2}$ sparse signal recovery in both noiseless and noise cases, compared with the state-of-the-art methods.

\section*{Acknowledgments} We would like to thank Dr Chao Wang for his discussion on the $\ell_{1}/\ell_{2}$$\_$box Matlab code.

\bibliographystyle{unsrtnat}
\bibliography{refer1}  






\end{document}